\documentclass[journal,twoside,web]{ieeecolor}
\usepackage{generic}
\usepackage{cite}
\usepackage{amsmath,amssymb,amsfonts}
\usepackage{algorithmic}
\usepackage{graphicx}
\usepackage{algorithm,algorithmic}
\usepackage{hyperref}
\hypersetup{hidelinks=true}
\usepackage{textcomp}
\usepackage{subfig}
\usepackage{booktabs}
\newtheorem{assm}{Assumption}

\newtheorem{lema}{Lemma}
\newtheorem{theorem}{Theorem}

\def\BibTeX{{\rm B\kern-.05em{\sc i\kern-.025em b}\kern-.08em
		T\kern-.1667em\lower.7ex\hbox{E}\kern-.125emX}}
\begin{document}

	\title{Switched Event-Triggered Adaptive  Control of Reaction-Diffusion PDE-ODE with Neural Operator Implementation}
	\author{Hongpeng Yuan, \IEEEmembership{Student Member, IEEE}, Ji Wang,\IEEEmembership{Member, IEEE} and Mamadou Diagne, \IEEEmembership{Senior Member, IEEE} 
		\thanks{Hongpeng Yuan and Ji Wang were supported by the National Natural Science Foundation of
China under Grant 62203372. The work of Mamadou Diagne was funded by the NSF CAREER Award CMMI-2302030 and  the NSF grant CMMI-2222250. The material in this paper was partially presented at The 64th IEEE Conference on Decision and Control(CDC), December 10--12, 2025, Rio de Janeiro, Brazil. {\it (Corresponding author: Mamadou Diagne.)}}
		\thanks{}
	}
	\maketitle

	\begin{abstract}

 This paper develops a switched event-triggered adaptive boundary control for a class of reaction-diffusion PDE-ODE cascade systems, where the system and input matrices in the ODE as well as the spatially-varying reaction coefficient in the PDE are uncertain. A two-step backstepping transformation is constructed to derive the continuous-time control law. Then a novel dynamic event-triggered control strategy for the PDE-ODE cascade is proposed based on a switched event-triggering mechanism, ensuring global exponential stability of the closed-loop system in place of the exponential convergence commonly achieved with backstepping-based classical dynamic ETC, while inherently excluding Zeno behavior. To address the uncertainties in the PDE-ODE cascade, adaptive update laws are developed, \textcolor{black}{leading to time-varying gain kernels that are adaptively scheduled through the event-triggered control mechanism. Furthermore,} to facilitate efficient real-time implementation, deep neural operators (DeepONet) are employed to approximate the backstepping kernels as mappings from the estimated parameters to kernel functions, thereby eliminating the need to repeatedly solve kernel PDEs online. Through a Lyapunov analysis that incorporates the effects of the event-triggering mechanism, parameter adaptation, and kernel approximation errors, we prove the $L^2$ global asymptotic regulation of the resulting closed-loop system. In summary, the key contributions of the paper are threefold: (i) developing an adaptive DeepONet-based framework for  reaction-diffusion PDE-ODE cascade systems; (ii) extending the existing adaptive event-triggered control design for reaction-diffusion PDEs to the case with more complex uncertainties; and (iii) generalizing switched dynamic ETC with global exponential stability to PDE-ODE cascades. The effectiveness of the proposed approach is demonstrated through numerical simulations.
	\end{abstract}
	

	\section{Introduction}\label{sec:Intro}
    \subsection{State-of the-art}
	\textbf{Boundary Control of Reaction-Diffusion PDEs and PDE-ODE cascades.}	
    Many engineering systems exhibit spatially distributed dynamics that are naturally modeled by partial differential equations (PDEs), making boundary control a fundamental tool for stabilization and regulation.
	 Among the broad class of PDE models, reaction-diffusion equations occupy a prominent place due to their ability to capture the interplay between spatial transport and local reaction dynamics. Such models arise naturally in a wide range of applications, including biology, ecology, physics, and chemistry (see e.g. \cite{article3}, \cite{bbook1}, \cite{article1}, \cite{article2}, \cite{article4}).  Important advances have been achieved in the control of reaction-diffusion PDEs, leading to a rich body of results on stabilization, estimation, and output regulation. In many practical settings, boundary actuation remains the most realistic and cost-effective mechanism for influencing the distributed dynamics.  
     
    The introduction of PDE backstepping represented a major advance in the control of distributed parameter systems by providing a systematic framework for boundary feedback design and stability analysis. Through invertible Volterra-type transformations, the method enables the explicit construction of stabilizing controllers and Lyapunov functionals \cite{doi:10.1137/1.9780898718607}. Since then, PDE backstepping has evolved into a powerful and versatile methodology capable of addressing increasingly complex settings, including \textcolor{black}{coupled linear reaction-diffusion systems \cite{Baccoli2015}, with linear actuator dynamics  \cite{Deutscher2020},  or scalar  nonlinear reaction-diffusion systems \cite{9580683} or their linear counterpart with nonlinear actuator dynamics \cite{Li2021}, among others. The framework has also proven highly effective for PDE–ODE cascade systems  including heat equations \cite{ANTONIOSUSTO2010284}, reaction-diffusion equations \cite{9978722}  and even  chains of PDE-ODE systems \cite{XU2023110763}.}

	\textbf{Event-Triggered Boundary Control of PDEs.}
	Event-triggered control (ETC) has been proposed as an effective mechanism to reduce communication and actuation updates by replacing periodic sampling with state-dependent triggering rules, e.g., \cite{6069816}, \cite{6310015}, \cite{6334427}, \cite{6924753}, \cite{10258286} \cite{9031377}. Originally studied for finite-dimensional systems, ETC has been extended to infinite-dimensional systems governed by PDEs.   Existing works have investigated ETC for various PDE systems under both state-feedback \cite{81107240}, \cite{ESPITIA20211095621}, \cite{9735290},  \cite{9978722}, and output-feedback settings   \cite{Espitiaoutput}, \cite{9477051},  \cite{YUAN2025112266}, \cite{Rathnayake2022Sampled}, \cite{9319184}, with applications including traffic flow control \cite{Espitiatraffic}, re-entrant manufacturing systems \cite{9483202}, and Stefan problems \cite{10156551}, \cite{Rathnayake2023Observer}. \textcolor{black}{Furthermore,} event-based control strategies  for PDEs have   been explored, including  periodic and self-triggered mechanisms  \cite{10572287}, \cite{SOMATHILAKE2025112433}, and performance-barrier event-triggered  control approaches    \cite{rathnayake2025performance}, \cite{10909456}.  Despite these advances, most existing results focus on ensuring exponential convergence while avoiding Zeno behavior. Results guaranteeing stronger stability properties, particularly for more complex systems such as PDE–ODE cascades, remain limited. Recently, global exponential stability (GES) results for ETC of PDEs have been established for both parabolic systems \cite{101093imamcidnaf042} and hyperbolic systems \cite{RATHNAYAKE2026112617}. 	 Besides PDE backstepping design,  observer-based modal decomposition methods for dynamic ETC of reaction-diffusion  PDEs are exploited in \cite{katz2020boundary, lhachemi2024event}. However, fewer results are available on PETC and/or STC \cite{selivanov2016distributed,wakaiki2020event,wakaiki2019stability}.
	
    
	\textbf{ Adaptive Boundary Control of Reaction–Diffusion PDEs.} Adaptive boundary control within the backstepping framework has been widely studied for  reaction–diffusion PDEs \cite{9978722}, \cite{4623267},   \cite{Li2021}, \cite{li2020adaptive} \cite{smyshlyaev2010adaptive,wang2021adaptive,wang2021delay,wang2024delay,wang2025adaptive}.   Recently, adaptive boundary control via neural operator-approximated gain kernels of a reaction-diffusion system with unknown spatially-varying reaction coefficient, was considered in  \cite{BHAN2025105968}.  
	For PDE–ODE cascade systems, event-triggered adaptive boundary control has also been considered in \cite{9978722}, where uncertainties arise in both the spatially invariant reaction coefficient of the PDE subsystem and the parameters of a scalar ODE.

	\textbf{Neural Operators for Backstepping Boundary Control of PDEs.} Neural operators have recently emerged as a powerful class of learning-based models for approximating mappings between infinite-dimensional function spaces, e.g., \cite{hornik1989multilayer}, \cite{chen1995universal}, \cite{lu2021learning}.	\textcolor{black}{Recent studies have explored neural network-based approximations of backstepping kernels. Primarily in settings with known function parameters \cite{krstic2024neural}, \cite{wang2025deep}, \cite{WANG2025112351}, which do not address adaptive or event-triggered implementations.}	In the context of adaptive boundary control, neural operators provide a promising approach to approximate parameter-dependent backstepping kernels, thereby avoiding repeated online solutions of kernel PDEs and significantly reducing computational complexity (see e.g. \cite{BHAN2025105968}, \cite{LAMARQUE2025112329}, \cite{LV2025112553}).

	\subsection{ Contributions}

    This paper addresses the boundary control for a class of uncertain reaction-diffusion PDE-ODE cascade systems and develops an event-triggered implementation with piecewise-constant control inputs. The PDE subsystem includes a spatially-varying reaction term that modifies the diffusion dynamics of the heat equation and may induce open-loop instability. The overall system is subject to heterogeneous boundary and interconnection conditions, comprising a Neumann boundary condition at the left endpoint, a Robin-type boundary control at the right endpoint, and a Dirichlet interconnection at the PDE-ODE interface. The high-order ODE subsystem is allowed to be unstable while remaining controllable. In addition, a deep neural operator (DeepONet) of the designed controller is implemented with provable stability result and arbitrary level of accuracy. The contributions of the present work are threefold.
    
    \begin{itemize}
     \item   \textbf{Generalization of adaptive DeepONet to complex RD PDE-ODE cascade systems.} Compared to \cite{BHAN2025105968}, which develops an adaptive DeepONet framework for reaction-diffusion PDEs with unknown spatially-varying reactivity and boundary coefficients, we address a more intricate system involving PDE–ODE cascade structures. In particular, both the spatially-varying reaction coefficient in the PDE subsystem and the system and input matrices of the higher-order ODE subsystem are assumed to be uncertain, leading to  \textcolor{black}{different parameter estimates} and control problem.

Furthermore, unlike \cite{BHAN2025105968}, the present work incorporates a switched dynamic event-triggered implementation, which introduces additional complexity into the adaptive law design. Specifically, the proposed update laws must simultaneously compensate for the unknown distributed reaction dynamics in the PDE actuation channel and the linearly parameterized system matrix and input vector of the ODE plant subsystem under discontinuous information exchange. These features collectively give rise to \textcolor{black}{challenges} in both controller synthesis and closed-loop stability analysis.

To the best of our knowledge, this is the first result establishing $L^2$-norm asymptotic regulation for a   uncertain PDE-ODE cascade system  under an event-triggered adaptive control framework \textcolor{black}{while} relying on neural-operator approximation of the closed-loop dynamics.

\item \textbf{\textcolor{black}{Generalization of  event-triggered adaptive control of reaction-diffusion PDEs.}} As the DeepONet architecture is employed to accelerate the computation of the controller gains, the proposed framework inherently extends  event-triggered adaptive control  of reaction-diffusion PDEs to a broader class of \textcolor{black}{ PDE-ODE cascade.} To the best of our knowledge, only the scalar counterpart of this problem has been addressed in \cite{9978722}, where the PDE subsystem involves a constant reaction coefficient and the ODE subsystem is scalar. The approach in \cite{9978722} relies on event-triggered Batch Least-Squares Identifiers (BaLSI), enabling exact and finite-time parameter identification together with exponential regulation of both the plant and actuator dynamics.

In the present work, however, we adopt a Lyapunov-based adaptive design due to the spatially-varying nature of the reaction coefficient, since the development of a BaLSI framework for PDEs with spatially distributed coefficients remains an open problem. Although the proposed Lyapunov-based approach neither guarantees exact and finite-time parameter identification nor exponential regulation--which are not the primary objectives of adaptive control--it achieves asymptotic regulation while generalizing the event-triggered adaptive boundary control framework of \cite{9978722} to systems with spatially-varying uncertainties in the PDE subsystem and simultaneous uncertainties in both the system and input matrices of high-order ODE subsystems.

It is worth emphasizing that event-triggered adaptive designs accommodating spatially-varying coefficients \textcolor{black}{remain unreported for hyperbolic PDE systems,} as evidenced by the existing literature represented by \cite{wang2022pde}.

\item \textbf{Generalization of  switched dynamic ETC with global exponential stability to PDE-ODE cascades.} In the present work, we obtain the first switched dynamic event-triggered controller for PDE–ODE cascade systems with spatially-varying coefficients,  broadening the results in \cite{101093imamcidnaf042}. Our design establishes a stronger stability result, namely, a global exponential stability of the closed-loop system, rather than the global exponential convergence typically obtained    for PDE-ODE cascade system based on PDE backstepping.

    \end{itemize}
 In summary, Our design involves three gain kernels, among which only the implicit and computationally demanding kernel is approximated through the DeepONet architecture. A preliminary version addressing switched and dynamic event-triggered adaptive control of reaction-diffusion PDEs with an unknown spatially-varying reaction coefficient is presented in our conference paper \cite{11312433}. In sharp contrast to \cite{11312433}, the present work considers a substantially more challenging PDE-ODE cascade configuration in which the high-order ODE subsystem is also subject to parametric uncertainty.

    The remainder of the paper is organized as follows. The problem formulation and the nominal continuous-in-time controller design are presented in Sections \ref{sec2} and \ref{sec3}, respectively. Section \ref{sec4} develops the nominal switched dynamic event-triggered control law. The Lyapunov-based adaptive control framework and its DeepONet-based approximation are presented in Section \ref{sec5}. Simulation results are provided in Section \ref{sec6}, followed by concluding remarks and future research directions in Section \ref{sec7}.

	\subsection{Notation}
	We adopt the following notation.
	\begin{enumerate}
		\item   $C^0(\Pi)$ and  $C^1(\Pi)$ are
		the usual notation for continuous and continuously differentiable
		functions (on the set  $\Pi$), respectively. By $C^k (U;\Omega)$, where $k \geq 1$, we denote the class of continuous functions on $U$, which takes values in $\Omega$ and has continuous	derivatives of order $k$ on $U$.  For scenarios where the function has multiple arguments, i.e., $f(x, y, t)$, we use $C_{x, y}^2 C_t^k$ to indicate the function has continuous second derivatives in $x$ and $y$, but only derivatives of order $k$ with respect to $t$.  
		
		\item The symbol $\mathbb{N}$ denotes the set of natural numbers including zero, and the notation $\mathbb{N}^*$ for the set $\{1,2, \cdots\}$, i.e., the natural numbers without 0. Let  $\mathbb{R}$ be the set of real numbers, and $\mathbb{R}_{+}:=[0,+\infty)$ represents the set of non-negative real numbers.
		
		\item  The partial derivatives and total derivatives are denoted as: $f_x(x, t)=\frac{\partial f}{\partial x}(x, t), f_t(x, t)=\frac{\partial f}{\partial t}(x, t)$, $f^{\prime}(x)=\frac{d f(x)}{d x}, \dot{f}(t)=\frac{d f(t)}{d t}$.
		
			\item  We use the notation $f[t]$ (e.g., $u,v$) to denote the profile of $f$ at certain $t \geq 0$, i.e., $(f[t])(x)=$ $f(x, t)$ for all $x\in [0,1]$.
			
			\item We use $\|  	\cdot \|_{\infty}$ for the infinity-norm.  Given an open set $\Omega$, $L^2(\Omega)=\{u: \Omega \rightarrow \mathbb{R} ; \int_{\Omega}|u(x)|^2 d x$ 
			$<+\infty\}$ 	endowed with the norm:
			$		\|u\|^2=\langle u, u\rangle=\int_{\Omega}|u(x)|^2 d x.$ The single bars $|\cdot|$ for $X(t)$ denote the Euclidean norm if $X(t) \in \mathbb{R}^n$.
	\end{enumerate}
	
	\section{Problem Formulation} \label{sec2}
	Consider the following plant
	\begin{align}
		& \dot{X}(t) = A(\theta) X(t) +B(\theta) u (0,t) , \label{0} \\
		& u_t(x, t)= u_{x x}(x, t)+\lambda(x) u(x, t), \label{1}\\
		& u_x(0, t)=0,\label{2} \\
		& u_x(1, t)+qu(1,t)= U(t) \label{3},
	\end{align}
	$\forall ( x,t) \in [0,1]\times[0,\infty)$, where $U(t)$ is the control input to be designed, $u(x,t)$ is the state of the reaction-diffusion PDE and $X(t) = [x_1,x_2,...,x_n]^T$ is the state of the ODE. The system matrix
	$A(\theta)$ and the input vector $B(\theta )$ are linearly parametrized, i.e.,
	\begin{align}
		&	A(\theta) = A_0 + \sum_{i=1}^{p} \theta_i A_i, \\
		& B(\theta) = B_0 + \sum_{i=1}^{p} \theta_i B_i,	\label{eq:B}	
	\end{align}
	and $\theta$ is an unknown but constant parameter vector that belongs to the convex set
	\begin{align}
		\Pi=\left\{\theta \in \mathbb{R}^p \mid  \mathcal{P}(\theta) \leq 0\right\} \label{Pithe}
	\end{align}
	where, by assuming that the convex function $\mathcal{P}: \mathbb{R}^p \rightarrow \mathbb{R}$ is smooth, we assure that the boundary $\partial \Pi$ of $\Pi$ is smooth.
	Further, $ \lambda(x)  : [0,1] \rightarrow \mathbb{R} $ is an unknown, spatially-varying coefficient function.
	
	To make stabilization  possible in the presence of unknown parameters, we make some assumptions and illustrate  that these assumptions are reasonable. 
	\begin{assm} \label{Assm1}
		The set ${\Pi}$ is bounded and known. 	There exists a constant $\bar{\lambda}>0$ such that $\|\lambda(x)\|_{\infty} \leq \bar{\lambda}$. 
	\end{assm}
	
	Assumption \ref{Assm1} about the known bounds of the unknown parameters is standard in the adaptive control  literature.
	\begin{assm} 
		We assume that the pair $(A(\theta), B(\theta))$ is completely controllable for each $\theta$. Furthermore, we assume that there exists a triple of vector/matrix-valued functions $(K(\theta), P(\theta), Q(\theta))$ such that $K \in C^1(\Pi), P \in C^1(\Pi), Q \in C^0(\Pi)$, the matrices $P(\theta)$ and $Q(\theta)$ are positive definite and symmetric, and the following Lyapunov equation is satisfied for all $\theta \in \Pi$ :
		\begin{align}
			P(\theta)(A+B K)(\theta)+(A+B K)(\theta)^{\mathrm{T}} P(\theta)=-Q(\theta) . \label{PQmatrix}
		\end{align}
	\end{assm}

 \begin{assm}\label{Assumption3} There exist known positive constants $\underline{\lambda_{\min}}$ and $\overline{\lambda_{\max}}$ such that \[ 
	\begin{aligned}
		&\min\left\{ \lambda_{\min}(P(\theta)), \lambda_{\min}(Q(\theta)) \right\} \geq \underline{\lambda_{\min}}, \\ 
		&\lambda_{\max}(P(\theta)) \leq \overline{\lambda_{\max}},
	\end{aligned}
	 \] for all \(\theta\in\Pi\). \end{assm}
    
	\begin{assm} \label{Assumption2}
		The parameters $q >0, \lambda \in C^2\left([0,1] ; \mathbb{R}_{+}\right)$ satisfy the following relation:
		\begin{align}
			q>\frac{\lambda_{\max }}{2  }+\frac{1}{2}, \label{qlt}
		\end{align}
		where
		$
		\lambda_{\max } \triangleq \max _{x \in[0,1]} \lambda(x) .$
	\end{assm} 
	According to Assumption \ref{Assumption2}, we can avoid using
	the signal $ u(1, t)$ in the nominal control law. Such avoidance is
	crucial for dynamic ETC design due to the challenges associated
	with obtaining a meaningful bound on the rate of change of $u(1, t)$.
	\section{Nominal continuous-in-time Control Design } \label{sec3}
	To systematically construct a stabilizing boundary feedback law, we employ a backstepping transformation that maps the original system into a target system with desirable stability properties.
	
	Consider the invertible backstepping transformation
	\begin{align}
		{w}(x, t)= &  {u}(x, t)-\int_0^x k(x, y)  {u}(y, t) \mathrm{d} y - \gamma(x) X(t), \label{wu}  
	\end{align}
	where $  \gamma(x )$  is the gain kernel   given by
	\begin{align}
		\textcolor{black}{\gamma(x)}=	& 
		K(\theta)\left[\begin{array}{ll}
			I & 0
		\end{array}\right] e^{\begin{array}{ll}
				{\left[\begin{array}{ll}
						0 & (A+BK)(\theta)\\
						I & 0
					\end{array}\right] x}
		\end{array} }
		\left[\begin{array}{l}
			I \\
			0
		\end{array}\right]\label{eq:gamma}
	\end{align}
	where the nominal gain $K$ is chosen to make $A+BK$ Hurwitz and $k(x,y)$ needs to satisfy
	\begin{align}
		&  k_{xx}(x,y)-  k_{yy}(x,y)-\lambda(y) k(x,y)=0,\label{k1} \\
		&k_y(x,0)=0,\\
		& k(x,x)=-\frac{1}{2} \int_0^x \lambda(y) dy.\label{k3}
	\end{align}
	The kernel equations \eqref{k1}--\eqref{k3} admit a unique solution on
	the triangular domain $\mathcal{T}={0 \leq y \leq x \leq 1}$ \cite{rathnayake2025performance}. Applying the backstepping transformation \eqref{wu} into the original system \eqref{0}--\eqref{3}  we arrive at the following system:
	\begin{align}
		\dot{X}(t)  = &  (A+BK) (\theta) X(t) + B(\theta) w (0,t),\label{wq1} \\
		{w}_t(x, t)=&   {w}_{x x}(x, t)-\gamma(x) B(\theta) w (0,t)  , \label{12} \\
		{w}_x(0, t)=&0\label{13}, \\
		{w}_x(1, t)=&U(t)-\wp w(1,t)- (\wp \gamma(1)+\gamma^{\prime}(1) )X(t)\notag\\
		&-\int_{0}^{1}(\wp k(1,y)+k_x(1,y)) u(y,t) dy , \label{eq18}
	\end{align}
	where 
	\begin{align}
		\wp = q-\frac{1}{2}\int_{0}^{1} \lambda(y)dy \label{eq019}
	\end{align}
	and  recalling \eqref{qlt} guaranteed in Assumption \ref{Assumption2} we derive
	\begin{align}
		\wp>\frac{1}{2}. \label{eq19}
	\end{align}
	
	In a method similar to the construction of the direct backstepping transformation, we obtain the inverse transformation of \eqref{wu} is given in the form of
	\begin{align}
		u(x,t) = w(x,t) + \int_{0}^{x} k^I(x,y) w(y,t) dy + \gamma^I(x) X(t) \label{uw}
	\end{align}
	where
	\begin{align}
		&	{k}^I_{x x}(x, y)-{k}^I_{y y}(x, y)  + {{\lambda}(x)}  {k}^I(x, y)=0, \\
		&{k}^I_y(x, 0) = 0 ,\\
		&{k}^I(x, x)  =-\frac{1}{2  } \int_0^x {\lambda}(y) d y, \\
		&\gamma^I(x)=\gamma(x)+\int_0^x k^I(x, s) \gamma(s) d s. 
	\end{align}
	
	Introduce the second transformation
	\begin{align}
		\beta(x,t) = w(x,t) - \int_0^x h(x,y) w(y,t) dy \label{bw}
	\end{align}
	where
	\begin{align}
		&  h_{xx}(x,y)-  h_{yy}(x,y) =0,\label{eq:h1}\\
		& h_x(x,x) + h_y(x,x) = 0,\label{eq:h2} \\
		&  \int_{0}^{x} \gamma(y) B(\theta) h(x,y) dy-\gamma(x) B(\theta)  - h_y(x,0) =0, \label{eq:h3}\\
		&h(0,0)=0,\label{eq:h4}
	\end{align}
	with choosing the control input $U(t)$ as a nominal continuous-in-time controller $U_c (t)$ defined by
	\begin{align}
		U_c (t) = & \int_{0}^{1} H(y; \lambda(y), \theta) u(y,t) dy + G(; \theta) X(t) \label{Uc}
	\end{align}
	where
	\begin{align}
		H(y; \lambda(y),\theta) = & \wp k(1,y)+k_x(1,y) +rh(1,y)+h_x(1,y) \notag\\
		&- \int_{y}^{1} k(z,y)(rh(1,z)+h_x(1,z) ) dz , \label{eq31}\\
		G(;\theta) = & -  \int_{0}^{1} (rh(1,y)+h_x(1,y) ) \gamma(y) dy \notag\\
		& +\wp \gamma(1)+\gamma^{\prime}(1),\label{eq32}
	\end{align}
	we arrive
	\begin{align}
		\dot{X}(t) =& (A+BK) (\theta) X(t) + B(\theta) \beta (0,t), \label{targ1} \\
		\beta_t(x,t) = & \beta_{xx} (x,t) , \label{targ3} \\
		\beta_x(0,t) = & 0, \label{targ2}\\
		\beta_x(1,t)  = & -r \beta(1,t) \label{targ4}
	\end{align}
	where
	\begin{align}
		r= \wp + h(1,1)=q+k(1,1)+h(1,1).  \label{eq33}
	\end{align}
	Using \eqref{eq:h2} and \eqref{eq:h4}, we can derive $h(1,1)=0$. Hence, \eqref{eq19} implies that
	\begin{align}
		r   > \frac{1}{2}. \label{eqr}
	\end{align}
	The detailed calculations regarding \eqref{bw}--\eqref{eq33} are shown in Appendix \ref{Appendix2}. The solution of $h(x,y)$ is  covered by the ones found in \cite{8281082} which ensures that \eqref{eq:h1}--\eqref{eq:h4} have a piecewise $C_2$-solution. \textcolor{black}{For know parameters, the global exponential stability of the target system \eqref{targ1}--\eqref{targ4} is provable following a Lyapunov argument, thus implying that of the original system under the continuous-in-time control law \eqref{Uc}. However, since our goal is to design an event-triggered adaptive control law by emulation of \eqref{Uc} and certainty equivalence, stability results will be established later on.} 
	
	According to \cite{9222479}, there exists kernel $h^I(x,y) \in \mathbb{R}$  for the
	inverse transformation of \eqref{bw}, which is in the form of
	\begin{align}
		w(x,t) = \beta(x,t) + \int_{0}^{x} h^I(x,y) \beta(y,t) dy.   \label{wb}
	\end{align}
	
	\section{Nominal Event-triggered Control Design} \label{sec4}
	\textcolor{black}{By emulation of}  the nominal continuous-in-time control design, the event-triggered control input is given as
	\begin{align}
		U_{d}(t) \ := & U_{c}(t_j) \notag\\
		=	& \int_{0}^{1} H(y; \lambda(y), \theta) u(y,t_j) dy+ G(;\theta) X(t_j)  .\label{c16}
	\end{align}
	for $t \in  [t_j , t_{j+1}), j \in  \mathbb{N}$, which is obtained by sampling the controller $U_{c}(t)$ \eqref{Uc} at a certain sequence of time instants $\left(t_j\right)_{j \in \mathbb{N}}$ that will be given based on an event trigger mechanism. The control signal \eqref{c16} is in a piecewise form, remaining constant between two consecutive time instants and being updated when a certain condition is met.  
    
	Define \textcolor{black}{the input holding error, namely,}  the difference between the continuous-in-time control signal $U_{c}(t)$ in \eqref{Uc} and the event-triggered control input $U_d(t) $
	in \eqref{c16} as $d(t)$, given by
	\begin{align}
		d(t)\ := & U_{c}(t) -U_d (t) ,\label{c170}
	\end{align}
	for $t \in\left[t_j, t_{j+1}\right)$, which will be used in building the event-triggered mechanism (ETM).\\
	The sequence of time instants $I=\left\{t_0, t_1, t_2, \ldots\right\}\ (t_0=0)$ is defined as (for $j\in \mathbb{N}$):\\
	\begin{equation}
		t_{j+1}=\inf \left\{t \geq t_j+ \tau \ | \ m(t) \leq 0 \right\},\label{c200}
	\end{equation}
	where the positive constant $\tau$ represents the minimal dwell time  and the dynamic variable $m(t)$ in \eqref{c200} satisfies the ordinary differential equation,
	\begin{align}
		\dot{m}(t)= & -\eta m(t) +\kappa_1\| {u}[t]\|^2 \label{c21}
	\end{align}
	for $t \in\left(t_j, t_{j}+ \tau \right)$, $j\in \mathbb{N}$, and 
	\begin{align}
		\dot{m}(t)= & -\eta m(t)-\lambda_d d(t)^2  +\kappa_1\| {u}[t]\|^2 \label{c211}
	\end{align}
	for $t \in \left(  t_{j}+ \tau, t_{j+1} \right)$, $j\in \mathbb{N}$ with $m\left(t_0\right)=m(0)=0$.   {\color{black}The design parameters $\eta,  \kappa_1 >0$ are free} and the positive design parameter  $ \lambda_d  $ is to be determined later. We define \begin{align}
		m((t_{j}+\tau)^{-})=m((t_{j}+\tau) )=m((t_{j}+\tau)^{+}) \label{mtdefine0}
	\end{align} for $j\in \mathbb{N}$ such that  $m(t)$ is continuous in $t\in [t_j, t_{j+1})$, $j\in \mathbb{N}$, and \begin{align}
		m(t_j)= \omega_0 d^2(t_j^-) \label{mtdefine10}
	\end{align} for $j\in \mathbb{N}^\star$, where $\omega_0$ is given later.
	
	\begin{lema}
		Consider the event-triggered mechanism of the form \eqref{c200}--\eqref{mtdefine10}, which generates a set of increasing event times $I=\left\{t_j\right\}_{j \in \mathbb{N}}$ with $t_0=0$, satisfying   $m(t) \geq 0$ for all $t>0$, and $\lim_{t \rightarrow \infty} t_j = \infty$, i.e., no Zeno behavior.
	\end{lema}
	\begin{proof}
		We have $m(t_0)=m(0)=0$, and from \eqref{c21} it follows that for $t\in [0,\tau]$
		\begin{align}
			m(t) = & m(t_j) e^{- \eta (t-t_j)} +\int_{t_j}^t e^{ \eta  (t- \tau )}  \kappa_1\| {u}[\tau]\|^2   d \tau ,
		\end{align}
		which indicates that $m(t) > 0$ for $t\in [0,\tau]$. Considering $m(t)$ is continuous in $t\in [0, t_{1})$, we can get $m(t_1^-)=0$  from \eqref{c200} such that $m(t)>0$ for $t\in [\tau, t_1)$.  Again by using $m(t_1)=\omega_0 d^2(t_1^-)>0$ and $m(t_2^-)=0$ we obtain $m(t)\geq 0$ for $t\in [t_1,t_2)$. By recursion it is derived that $m(t) \geq 0$ all the time. 

	Inspired by \cite{101093imamcidnaf042}, we consider a piecewise right continuous function $f(t)  $ satisfies	
	\begin{align}
		\dot{f}(t)= \begin{cases}-a_2 f^2(t)-a_1 f(t)-a_0, & \forall t \in\left(t_j, t_j+\tau\right), \\ 0, & \forall t \in\left(t_j+\tau, t_{j+1}\right),\end{cases} \label{f(t)}
	\end{align}	
	for $j \in \mathbb{N}$, where $\tau>0$ is the minimal dwell time, and $a_0, a_1, a_2>0$, with $f\left(t_j\right)=\omega_1>0$ and $f\left(\left(t_j+\tau\right)^{-}\right)=f\left(t_j+\tau\right)= \omega_0>0$ and $\omega_0<\omega_1$. This yields that $\omega_0 \leq f(t) \leq \omega_1$. Specifically, the function $f(t)$ reduces from $ \omega_1$ to $ \omega_0$ for all $t \in\left[t_j, t_j+\tau\right)$, and remains   $ \omega_0$ for all $t \in\left[t_j+\tau, t_{j+1}\right)$. The minimal dwell time $\tau$ is	given by
	\begin{align}
		\tau=\int_{\omega_0}^{\omega_1} \frac{1}{a_2 s^2+a_1 s+a_0} d s. \label{taut0}
	\end{align}	
	Thus,  the avoidance of the Zeno phenomenon in the ETC design  is thus obtained.	
		\end{proof}
	\begin{theorem} \label{theorem1}
		For all initial conditions $u[0] \in L^2(0,1)$  and $X(0) \in \mathbb{R}^n$, the closed-loop system,  which consists of the plant \eqref{0}--\eqref{3} and the event-triggered control law \eqref{c16}  with the event-triggering
		mechanism \eqref{c200}--\eqref{mtdefine10}, is globally exponentially stable in the sense that there exist positive constants $M_1, \sigma$ such that
		\begin{align}
			\Omega(t) \leq M_1 \Omega(0) e^{-\sigma t}, \label{ome}
		\end{align}
		where
		\begin{align}
			\Omega(t) =  \|u[t]\|^2 + |X(t)|^2+ m(t) +f(t) d^2(t). \label{Omega}
		\end{align}
	\end{theorem}
	\begin{proof}
		It can be shown that applying the backstepping transformations \eqref{wu} and \eqref{wb} into the system \eqref{0}--\eqref{3} with choosing the control law  as $U_d$ given in \eqref{c16}, yields the target system \eqref{targ1}--\eqref{targ2} and  the other boundary condition
		\begin{align}
			\beta_x(1,t)  =  -r \beta(1,t) - d(t)  \label{wt3}
		\end{align}
		for $t \in\left[t_j, t_{j+1}\right), j \in \mathbb{N}$.
		
		Define a Lyapunov function as
		\begin{align}
			V(t)=&  \frac{r_a}{2} \|\beta[t]\|^2 + {r_b} X^T(t) P(\theta) X(t) +  m(t)  +  f(t) d^2(t)   \label{lyapunov}
		\end{align}
		where $ r_a,r_b $ are positive constants to be chosen later and $P(\theta)$ is a positive definite matrix  to satisfy \eqref{PQmatrix}. Using \eqref{mtdefine0}, \eqref{mtdefine10} and the definition of $f(t)$ given before, we can prove that $m(t)+f(t) d^2(t)$ is continuous and thus $V(t)$ is continuous as well.  Considering the target system \eqref{targ1}--\eqref{targ2}, \eqref{wt3} and  using \eqref{c21},   we obtain  for $t \in\left(t_j, t_{j}+\tau \right)$ 
		\begin{align}
			\dot{V}  = &	r_a  \big( - \|\beta_x \|^2 - r  \beta^2(1,t)  - \beta(1,t)  d(t) \big)\notag\\
			& + r_b \big(- X^T(t) Q(\theta) X(t) +2 X^T(t) P(\theta)  B(\theta) \beta(0,t)  \big)  \notag\\
			&    -\eta m(t) +\kappa_1\| {u}[t]\|^2  +2f(t)d(t)\dot{d}(t) + \dot{f}(t) d^2(t)   .   \label{Vlya}
		\end{align}
		The event-triggered control input $U_d(t)$ is constant on $t \in\left(t_j, t_{j+1}\right)$, i.e., $\dot{U}_d(t)=0$. Taking the time derivative of \eqref{c170}, recalling  \eqref{Uc}--\eqref{eq32}, applying integration by parts, we obtain that
		\begin{align}
			\dot{d}(t)	=&\underline{b_1 } d(t)  +\underline{b_2} u(0,t)  +\underline{b_3} u(1,t) \notag\\
			&+ \int_{0}^{1} \underline{b_4}(y) u(y,t)dy   +\underline{B_5} X(t)
			\label{c18}
		\end{align}
		for $t\in (t_j,t_{j+1})$, where
		\begin{align}
			\underline{b_1 }  = &  {H}( 1 ) ,\label{vn1}\\
			\underline{b_2} = &  {G} B(\theta)  -  {H}_y( 0 )  ,\\
			\underline{b_3}  =&  -  {H}_y( 1 )-q {H}(1 ),\\
			\underline{b_4} (y )= 	&  {H}_{yy} (y ) +  {H}  (y ) \lambda(y)  + {H}( 1 )  {H} ( y ) ,\\
			\underline{B_5} = &     {G}  A(\theta) + {H}( 1 )  {G} .
		\end{align}
		Applying the Cauchy-Schwarz inequality into \eqref{c18}, we then obtain that
		\begin{align}
			\dot{d}(t)^2 \leq &\epsilon_1 d(t)^2+\epsilon_2  {u}^2(0,t)  +\epsilon_3 {u}^2(1,t) +\epsilon_4 \|u[t]\|^2  \notag\\
			&+\epsilon_5 |X(t)|^2  \label{ddt1}
		\end{align} 	for $t \in\left(t_j, t_{j+1}\right)$, where	\begin{align}
			&\epsilon_1= 5  \underline{b_1}^2,  
			\ \epsilon_2= 5  \underline{b_2}^2 ,\ \epsilon_3= 5 \underline{b_3}^2, \notag \\
			&	\epsilon_4=5  \underline{b_4} ^2 ,\
			\epsilon_5=5| \underline{B_5}|^2 .\label{e4}
		\end{align}
		To simplify notation, define the following constants for the bounds for both the   backstepping kernels and inverse backstepping kernels as $\bar{f} := \|f\|_{\infty}$ for $f=k,k^I,\gamma,\gamma^I,h,h^I$.
		The inverse transformations \eqref{uw} and \eqref{bw},  together with the Cauchy–Schwarz and Young's inequalities, yield
		\begin{align}
			u^2(1,t) \leq & 3  {w}^2(1,t)+ 3 \bar{k^I}^2 \| {w}[t]\|^2 + 3 \bar{\gamma^I}^{2}  |X(t)|^2 ,\label{eq63} \\
			u^2(0,t) \leq  & 2  {w}^2(0,t)  + 2 \bar{\gamma^I}^{2}  |X(t)|^2, \\
			{w}^2(1,t) \leq & 2  {\beta}^2(1,t)  + 2 \bar{h^I}^2 \| {\beta}[t]\|^2 ,\\
			{w}^2(0,t) =  &   {\beta}^2(0,t),\\
			\|u[t]\|^2 \leq & 2 (1+\bar{k^I})^2 \| {w}[t]\|^2 + 2 \bar{\gamma^I}^{2} |X(t)|^2   , \\
			\| {w}[t]\|^2  \leq & (1+\bar{h^I})^2 \| {\beta}[t]\|^2 . 
		\end{align}
		For the $\beta$-subsystem and $w$-subsystem, using  Agmon's  inequality, the following inequalities hold:
		\begin{align}		 
			{w}^2(0,t) \leq&    {w}^2(1,t)+\|  {w}_x[t]\|^2+\|  {w}[t]\|^2,\label{v3}  \\
			{\beta}^2(0,t) \leq&   {\beta}^2(1,t)+\|  {\beta}_x[t]\|^2+\|  {\beta}[t]\|^2. \label{eq148}
		\end{align}
		Therefore, using \eqref{eq63}--\eqref{eq148}, we obtain
		\begin{align}
			u^2(0,t) \leq  & 2  ( {\beta}^2(1,t)+\|  {\beta}_x[t]\|^2+\|  {\beta}[t]\|^2 +   \bar{\gamma^I}^{2}  |X(t)|^2),   \label{eq145} \\
			u^2(1,t) \leq & 6  {\beta}^2(1,t)+  6 \bar{h^I}^2  \| {\beta}[t]\|^2 + 3 \bar{k^I}^2  (1+\bar{h^I})^2 \| {\beta}[t]\|^2  \notag\\
			& + 3  \bar{\gamma^I}^{2}  |X(t)|^2 , \label{eq144} \\
			\|u[t]\|^2 \leq &   2 (1+\bar{k^I})^2 (1+\bar{h^I})^2 \| {\beta}[t]\|^2 + 2 \bar{\gamma^I}^{2} |X(t)|^2  . \label{equation70}
		\end{align}
		For the $\beta$-subsystem, using Poincare's inequality, we have that
		\begin{align}
			-\| {\beta}_x[t]\|^2  \leq & \frac{1}{2}  {\beta}^2(1,t) - \frac{1}{4} \| {\beta}[t]\|^2. \label{eq74}
		\end{align}
		  	Using \eqref{Vlya},   \eqref{c21}, \eqref{f(t)}, \eqref{ddt1}  and  \eqref{eq148}--\eqref{eq74}, the following estimate holds:
		\begin{align}
			\dot{V} \leq&    -\mu_1 \| \beta[t]\|^2- \mu_2   \|\beta_x[t]\|^2-\mu_3 \beta^2(1,t) -\mu_4 |X(t)|^2 \notag\\
			&-\mu_5 d(t)^2 -\eta m(t) 
		\end{align}
		for $t \in\left(t_j, t_{j}+\tau \right)$, where
		\begin{align}
			\mu_1 = &\frac{r_a}{8} - \frac{\epsilon_3}{a_2} (3 \bar{k^I}^2  (1+\bar{h^I})^2 +  6 \bar{h^I}^2)   -\frac{r_b}{4 \delta_2}\notag\\
			&-  \frac{2 \epsilon_2}{a_2}- 2 (1+\bar{k^I})^2 (1+\bar{h^I})^2 (\kappa_1 +  \frac{\epsilon_4}{a_2}), \label{vmu1} \\
			\mu_2 = &\frac{r_a}{2}-    \frac{2 \epsilon_2}{a_2}-\frac{r_b}{4 \delta_2},  \\
			\mu_3 = & {(r-\frac{1}{4}-\delta_1)} r_a -   \frac{2 \epsilon_2}{a_2}   -  \frac{6\epsilon_3}{a_2} -\frac{r_b}{4 \delta_2}, \\
			\mu_4 = & r_b \lambda_{\min} (Q) -2 r_b \delta_2 \lambda_{\max} (PBB^TP^T)  -  \frac{3  \bar{\gamma^I}^{2}  \epsilon_3}{a_2} \notag\\
			& -2 \bar{\gamma^I}^{2} ( \frac{\epsilon_2}{a_2} +\kappa_1 +  \frac{\epsilon_4}{a_2}), \label{vmu4} \\
			\mu_5= &a_1 f(t) +a_0-\frac{\epsilon_1}{a_2}-\frac{r_a  }{4 \delta_1}.
		\end{align}
		where $\delta_1, \ \delta_2$ are constants to be chosen later.  Choosing
		\begin{align}
			\delta_2 < & \frac{\lambda_{\min}(Q)}{2\lambda_{\max}(PBB^TP^T)} , \\
			r_b > & \frac{    \frac{3  \bar{\gamma^I}^{2}  \epsilon_3}{a_2} +2 \bar{\gamma^I}^{2} (  \frac{\epsilon_2}{a_2} +\kappa_1 +  \frac{\epsilon_4}{a_2}) }{\lambda_{\min}(Q)- 2\delta_2 \lambda_{\max} (PBB^TP^T)} ,\\
			r_a > & \frac{ 16  \epsilon_2}{a_2}  +\frac{8 \epsilon_3}{a_2}  (3 \bar{k^I}^2  (1+\bar{h^I})^2 +  6 \bar{h^I}^2)   \notag\\
			&+16 (1+\bar{k^I})^2 (1+\bar{h^I})^2 (\kappa_1 +  \frac{\epsilon_4}{a_2})+\frac{2 r_b}{ \delta_2} , \\
			\delta_1 < &  \frac{(r-\frac{1}{4})r_a-  \frac{2 \epsilon_2}{a_2} -   \frac{6\epsilon_3}{a_2} -\frac{r_b}{4 \delta_2}}{r_a}   ,\\
			a_0 \geq &  \frac{\epsilon_1}{a_2}+\frac{r_a  }{4 \delta_1} , 
		\end{align}
		such that $\mu_1, \mu_2, \mu_3, \mu_4 , \mu_5$ are positive constants and 
		\begin{align}
			\dot{V} \leq & - \sigma_1  V(t) \label{eq80}
		\end{align}
		for $t \in\left(t_j, t_{j}+\tau\right)$,	where 
		\begin{align}
			\sigma_1  = \min \{ \frac{2 \mu_1}{r_a}, \ \frac{\mu_4 }{r_b \lambda_{\max}(P)},\ \eta, \ \frac{\mu_5}{\omega_1} \}.
		\end{align}
		Similarly, 	using  \eqref{Vlya},  \eqref{c211}, \eqref{f(t)}, \eqref{ddt1}  and  \eqref{eq148}--\eqref{eq74}, we get
		\begin{align}
			\dot{V} \leq&    -\mu_1 \| \beta[t]\|^2- \mu_2   \|\beta_x[t]\|^2-\mu_3 \beta^2(1,t) -\mu_4 |X(t)|^2 \notag\\
			&  -\mu_6 d(t)^2-\eta m(t)
		\end{align}
		for $t \in\left(t_j+\tau , t_{j+1}  \right)$, where $\mu_1, \ \mu_2, \ \mu_3, \ \mu_4$ are given in \eqref{vmu1}--\eqref{vmu4} and 
		\begin{align}
			\mu_6 = \lambda_d-a_2 f^2(t)-\frac{\epsilon_1}{a_2}-\frac{r_a  }{4 \delta_0} .
		\end{align}
		Recalling the fact that $\omega_0 \leq f(t) \leq \omega_1$,   choosing
		\begin{align}		
			\lambda_d \geq  a_2 \omega_1^2+ \frac{\varepsilon_1}{a_2} +\frac{r_a }{4 \delta_0}  , 
		\end{align}
		we can derive 
		\begin{align}
			\dot{V} \leq & - \sigma_2  V(t) \label{eq801}
		\end{align}
		for $t \in\left( t_{j}+\tau, t_{j+1} \right)$, where the positive constant $\sigma_2$ is given by
		\begin{align}
			\sigma_2  = \min \{ \frac{2 \mu_1}{r_a}, \ \frac{\mu_4 }{r_b \lambda_{\max}(P)},\ \eta, \ \frac{\mu_6}{\omega_1} \}.
		\end{align} 
		Let
		\begin{align}
			\sigma = \min \{ \sigma_1, \ \sigma_2 \}.
		\end{align}
		Then,
		\begin{align}
			V\left(t_{j+1}\right) \leq & e^{-\sigma\left(t_{j+1}-(t_j+\tau )\right)} V\left(t_j+\tau\right) \notag\\
			\leq & e^{-\sigma\left(t_{j+1}-(t_j+\tau )\right)}e^{-\sigma \tau   }V\left(t_j\right).
		\end{align}
		Hence, for any $t \geq 0$ in $t \in\left[t_j, t_{j+1}\right), j \in \mathbb{N}$,  by recursion we derive:
		\begin{align}
			V(t)  \leq e^{-\sigma \left(t-t_j\right)} V\left(t_j\right) \leq e^{-\sigma t} V(0). \label{Vte}
		\end{align}
		The inverse transformation $\beta \mapsto u$ is given by 
		\begin{align}
			u(x,t) = &  {\beta}(x,t) + \int_{0}^{x} \check{h}^I(x,y) \hat{\beta}(y,t) dy + \check{\gamma}^I(x,t) X(t)  \notag\\
			& + \int_{0}^{x} \hat{k}^I(x,y) \big(\hat{\beta}(y,t) + \int_{0}^{y} \check{h}^I(y,z) \hat{\beta}(z,t) dz \big) dy,  
		\end{align}
		Recalling  \eqref{Omega}, \eqref{lyapunov} and using  \eqref{wu}, \eqref{bw} and \eqref{equation70}, we can derive
		\begin{align}
			\Omega(t) \leq& M_1  e^{-\sigma t} \Omega(0)
		\end{align}
		where
		\begin{align}
			M_1 = & 4 \max \{(1+\bar{k^I})^2 (1+\bar{h^I})^2, \bar{\gamma^I}^2 \} \times \notag\\
			& \ \  \max \{(1+\bar{k})^2 (1+\bar{h})^2, \bar{\gamma}^2  \}.
		\end{align}
	\end{proof}
We present the design steps and results in the nominal ETC design in Fig. \ref{figkernel}.
	\begin{figure}
	\hspace{0.1cm}
	\includegraphics[width=9cm]{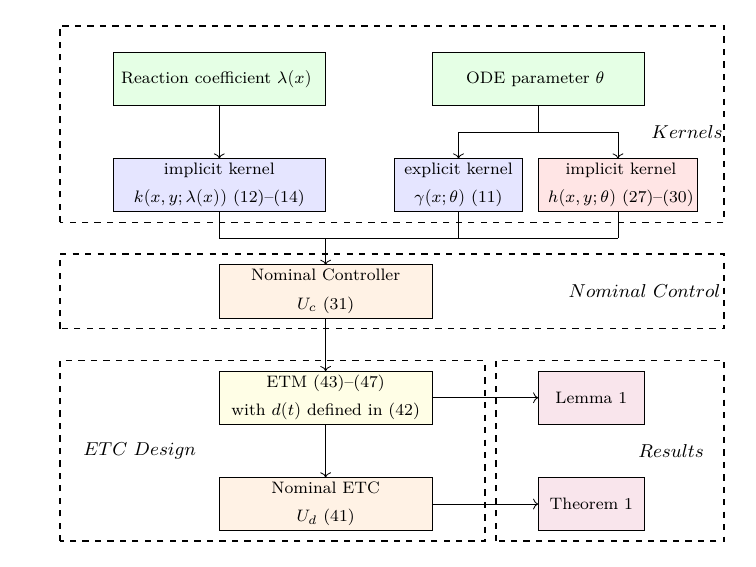} 	
	\caption{ The design process and results under nominal event-tiggered controlwith known parameters.} 
	\label{figkernel}
\end{figure}
	\section{Neural \textcolor{black}{operator}-based Adaptive Event-triggered Control Design} \label{sec5}
	\textcolor{black}{In this section, we develop a neural operator-based  event-triggered adaptive backstepping control design. Following a certainty equivalence principle, we  extend the ETC control methodology in the previous section to the case where $\theta$ and $\lambda(x)$ are unknown.} We build adaptive update laws $\hat{\lambda}$ and $\hat{\theta}$ for the unknown  parameters, and introduce the neural operator using DeepONet to approximate the mapping from
	the estimated system parameters to the backstepping kernels, so as to facilitate real-time implementation of the adaptive controller.
	\subsection{ \textcolor{black}{Adaptive} Control Design}
	Unlike the nominal backstepping design for known-parameter systems, the three transformation kernels $k(x,y), \ \gamma(x), \ h(x,y)$ in the present setting depend  on the unknown reaction coefficient, and therefore become  time-varying under adaptive estimation. In what follows, we define the kernel funcitons derived by estimated parameters as:
	\begin{align}
		\check{k}(x,y,t):= & k(x,y; \hat{\lambda}(t)), \label{checkk} \\
		\check{\gamma}(x,t) := & \gamma(x; \hat{\theta}(t)) ,\label{checkg} \\
		\check{h}(x,y,t) := & h(x,y; \hat{\theta}(t)).\label{eq:checkh}
	\end{align}
	In this section, we redefine   $\bar{f} := \|f\|_{\infty}$ for $f= \check{\gamma},\check{\gamma}^I,\check{h},\check{h}^I$.
	We choose the  update law for
	the estimate $\hat{\lambda}$ as 
	\begin{align} 
		\hat{\lambda}_t(x, t):= & \operatorname{Proj}_{0}(\hat{\phi}_{\lambda}(x, t), \hat{\lambda}(x, t)), \label{eventAdaptivelawlambda} \\
		\hat{\phi}_{\lambda}(x, t):= & \frac{r_1 u(x, t)}{1+\Upsilon (t)}  \left(\hat{\beta}(x, t)-\int_x^1 \hat{k}(y, x, t) \hat{\beta}(y, t) dy\right)  \label{evl}
	\end{align}
	and update law for the estimate $\hat{\theta}_i$ as (for $i \in \mathbb{N}^{\star}$)
	\begin{align}
		\dot{\hat{\theta}}_i(t):= & \operatorname{Proj}_{i}(\hat{\phi}_{\theta_i}( t),\theta_i(t)),\label{equation105} \\
		\hat{\phi}_{\theta_i}( t):= &  \frac{1}{1+\Upsilon (t)}  \bigg( r_{3}\big( -\int_{0}^{1} \hat{\beta}(y,t) \check{\gamma}(y,t) dy A_i X(t) \notag\\
		& -\int_{0}^{1} \hat{\beta}(y,t) \check{\gamma}(y,t) dy B_i u(0,t) \big)  \notag\\
		& +  2 r_2 \big( {  X^T(t) P(\hat{\theta}) A_i X(t)+  X^T(t) P(\hat{\theta}) B_i u(0,t) } \big) \bigg)  \label{eventAdaptivelawtheta}
	\end{align}
	where positive constants  $r_1, \ r_2, \ r_3$ are  the adaptation gains   and $\Upsilon(t)$ is a function to be given by \eqref{Upsilon} in Lyapunov analysis part. The neural \textcolor{black}{operator}-based kernel function $\hat{k}$ is given later.  The projections are defined as 
	\begin{align}
		&\operatorname{Proj}_{0}(a, b):= \begin{cases}0, & \text { if }|b| \geq \bar{\lambda} \text { and } a b>0 \\ a , & \text { otherwise }\end{cases} \label{equation107} \\
		&\operatorname{Proj}_i(a, b):= \begin{cases}0, & \text { if }|b| \geq \bar{\theta_i} \text { and } a b>0 \\ a , & \text { otherwise. }\end{cases}  \label{equation110}
	\end{align}
	where the known bounds $\bar{\lambda}$, $\bar{\theta_i}$ of  the unknown parameters $\lambda$ and $\theta_i $ for $i \in \mathbb{N}^* $,  $i \leq p$  are given by Assumption \ref{Assm1}. The projection operator $\operatorname{Proj}_0$ is to keep the estimated reaction coefficient $\hat{\lambda}$ bounded within $\bar{\lambda}$. The projection operators $\operatorname{Proj}_i$ are to keep the scalar components of
	parameter estimate vector $\hat{\theta}(t) = [\hat{\theta}_1, ... , \hat{\theta}_p ]$ bounded within
	$\bar{\theta_i}$. 
	
	We define the parameter estimate errors $\tilde{\lambda}(x,t), \tilde{\theta}(t)$ as 
	\begin{align}
		\tilde{\lambda}(x,t) =& \lambda(x) -\hat{\lambda}(x,t) ,\\
		\tilde{\theta}(t) = &\theta -\hat{\theta}(t). 
	\end{align}
	Define the set of functions
	\begin{align}
		\underline{K}=\left\{k \in C_{x, y}^2 C_t^1\left(\mathcal{T} \times \mathbb{R}_{+}\right)  \right\}, \quad \forall x \in[0,1], t \in \mathbb{R}_{+} .
	\end{align}
	Let $\Lambda$ be a compact set of $C^2([0,1])$.  
	The kernel operator $\mathcal{K}: \Lambda \rightarrow \underline{K} $ is defined by
	\begin{align}
		\mathcal{K} (\hat{\lambda}(\cdot,t)) :=  \check{k}(x, y, t)   . \label{calk}
	\end{align}
	The kernel operator $\mathcal{K}$ maps the estimated system parameters to the backstepping kernels, such that there exists a neural operator $\hat{\mathcal{K}}$ approximates the kernel operator $\mathcal{K}$.  For the subsequent Lyapunov analysis, the approximation of $\check{k}_t$ is also required. Thus, define the operator $\mathcal{K}_1: \Lambda^2 \rightarrow C_{x, y}^2 C_t^0\left(\mathcal{T} \times \mathbb{R}_{+}\right)$ such that
	\begin{align}
		\mathcal{K}_1\left(\hat{\lambda}(\cdot, t), \hat{\lambda}_t(\cdot, t)\right):=\check{k}_t(x, y, t).
	\end{align}
	According to \cite{BHAN2025105968} and \cite{doi:10.1137/1.9780898718607},  $\mathcal{K}_1$ is Lipschitz continuous. We then have the following lemma:
	\begin{lema} \label{lema1}
		For all $\iota>0$, there exists a neural operator $\hat{\mathcal{K}}$ such that  for all   $t\geq0$ and $(x, y) \in \mathcal{T}$ 
		\begin{align}
			& |\mathcal{K}(\hat{\lambda})(x, y,t)-\hat{\mathcal{K}}(\hat{\lambda})(x, y,t)| \notag\\
			& +\left|2 \frac{d}{d x}(\mathcal{K}(\hat{\lambda})(x, x,t)-\hat{\mathcal{K}}(\hat{\lambda})(x, x,t))\right| \notag\\
			& +|\left(\partial_{x x}-\partial_{y y}\right)(\mathcal{K}(\hat{\lambda})(x, y,t)-\hat{\mathcal{K}}(\hat{\lambda})(x, y,t)) \notag\\
			& -\hat{\lambda}(y,t)(\mathcal{K}(\hat{\lambda})(x, y,t)-\hat{\mathcal{K}}(\hat{\lambda})(x, y,t)) | \notag \\
			&+ |\mathcal{K}_1(\hat{\lambda},\hat{\lambda}_t)(x, y,t)-\frac{d}{d t}\hat{\mathcal{K}}(\hat{\lambda} )(x, y,t)| <\iota.\label{supk}
		\end{align}	
	\end{lema}
	\begin{proof}
		The system \eqref{k1}-\eqref{k3} has a unique $ {C}^2(\mathcal{T})$ solution, therefore the neural operator $\hat{\mathcal{K}} (\hat{\lambda} )(x, y,t)$ could approximate the kernels for estimated parameters. Using DeepONet universal approximation theorem (\cite{deng2022approximation}, Theorem 2.1), we can obtain a maximum approximation error defined as $\iota$. Then similar to (\cite{BHAN2025105968}, Theorem 3), we can finish the proof of Lemma \ref{lema1}.
	\end{proof}
	The neural \textcolor{black}{operator}-based adaptive control law is defined as
	\begin{align}
		U_{N}(t) =&  \int_{0}^{1} \hat{H}(y; \hat{\lambda}(y,t), \hat{\theta}(t)) u(y,t) dy + \check{G}(; \hat{\theta}(t)) X(t)\label{NOcontrol}
	\end{align}	
	where 
	\begin{align}
		\hat{H}(y; \hat{\lambda}(y,t),\hat{\theta}(t)) = & \wp \hat{k}(1,y)+\hat{k}_x(1,y) +r\check{h}(1,y)+\check{h}_x(1,y) \notag\\
		&- \int_{y}^{1} \hat{k}(z,y)(r\check{h}(1,z)+\check{h}_x(1,z) ) dz , \label{eq312}\\
		\check{G}(;\hat{\theta}(t)) = &-  \int_{0}^{1} (r\check{h}(1,y)+
		\check{h}_x(1,y) ) \check{\gamma}(y) dy \notag\\
		&+\wp \check{\gamma}(1)+\check{\gamma}^{\prime}(1)  \label{eq322}
	\end{align}
	with \textcolor{black}{neural operator} gain kernels $\hat{k}= \hat{\mathcal{K}}(\hat{\lambda})$. The signal $U_N$ does not act as the control input
	of the plant but used in the ETM. We define
	the event-triggered-type neural operator-based adaptive control input for $t \in  [t_j , t_{j+1}), j \in  \mathbb{N}$, as
	\begin{align}
		U_{a}  := & U_{N}(t_j) \notag\\
		=	& \int_{0}^{1} \hat{H}(y,t_j) u(y,t_j) dy + \check{G}( t_j) X(t_j).\label{c162}
	\end{align}
	Define the difference between the adaptive continuous-in-time control signal $U_{N}(t)$ in \eqref{NOcontrol} and the event-triggered control input $U_a(t)$
	in \eqref{c162} as $d(t)$, given by
	\begin{align}
		d(t)\ := & U_{N}(t) -U_a(t)   \label{c1701}
	\end{align}
	for $t \in\left[t_j, t_{j+1}\right)$. 
\begin{lema} \label{lemareMDT}
		For the neural \textcolor{black}{operator}-based adaptive event-triggered controller, consider the event-triggered mechanism \eqref{c200}--\eqref{mtdefine10} with $d(t)$ in \eqref{c211} and \eqref{mtdefine10} redefined according to \eqref{c1701}. Then, the Zeno phenomenon is excluded.
	\end{lema}
	
	\begin{proof}
		The proof follows from the same argument used to establish \eqref{f(t)} and \eqref{taut0}, since the modified definition preserves the required properties.
	\end{proof}
    
	Define the kernel approximate error $\tilde{k}(x,y,t)$ as 
	\begin{align}
		\tilde{k}(x,y,t)=&\check{k}(x,y,t)-\hat{k}(x,y,t).
	\end{align}
	According to \cite{1369395},  the following bound is guaranteed:
	\begin{align}
		\|  {k}(x,y,t) \|_{\infty} \leq  2 \|\lambda\|_{\infty} \mathrm{e}^{  4 \|\lambda\|_{\infty}  } . \label{42q1}
	\end{align}
	Using \eqref{42q1} and recalling \eqref{checkk}, \eqref{eventAdaptivelawlambda} and \eqref{equation107},  we can get
	\begin{align}
		\| \check{k}(x,y,t) \|_{\infty}& \leq  {2\bar{\lambda}} \mathrm{e}^{  4 \bar{\lambda}  } . \label{42q}
	\end{align}
	From \eqref{NOcontrol}--\eqref{c1701}, it follows that
	\begin{align}
		\dot{d}(t)	=&b_1(t) d(t)  +b_2(t) u(0,t)  +b_3(t) u(1,t) \notag\\
		&+ \int_{0}^{1} b_4(y,t) u(y,t)dy   +B_5(t) X(t)
		\label{c182}
	\end{align}
	where
	\begin{align}
		b_1(t) = & \hat{H}( 1,t) ,\label{vn11}\\
		b_2(t) = & \check{G}(t)B( {\theta})  - \hat{H}_y( 0,t )  ,\\
		b_3(t) =&  - \hat{H}_y( 1,t )-q\hat{H}(1,t ),\\
		b_4 (y,t)=& \hat{H}_{ t} (y,t ) +  \hat{H}_{yy} (y,t )  +  \hat{H}  (y,t ) \lambda(y)\notag\\
		&  + \hat{H}( 1 ,t)  \hat{H} ( y,t ) ,\\
		B_5(t)= &  \check{G}_{t}(  t )  +  \check{G}( t )  A(\theta) +  \hat{H}( 1 ,t) \check{G} (t).
	\end{align}
	Applying the Cauchy-Schwarz inequality into \eqref{c18}, we then obtain 
	\begin{align}
		\dot{d}(t)^2 \leq &\epsilon_1 d(t)^2+\epsilon_2  {u}^2(0,t)  +\epsilon_3 {u}^2(1,t) +\epsilon_4 \|u[t]\|^2  \notag\\
		&+\epsilon_5 |X(t)|^2  \label{ddt}
	\end{align} 	for $t \in\left(t_j, t_{j+1}\right)$, where	\begin{align}
		&\epsilon_1= 5 \bar{b}_1^2,  
		\ \epsilon_2= 5 \bar{b}_2^2 ,\ \epsilon_3= 5\bar{b}_3^2, \notag \\
		&	\epsilon_4=5 \bar{ b}_4 ^2 ,\
		\epsilon_5=5|\bar{B}_5|^2 \label{e41}
	\end{align}
	with $\bar{b}_i= \max_{t\in (t_j,t_{j+1})} \{ |b_i(t)|\}$ for $i=1,2,3$, $\bar{b}_4= \max_{t\in (t_j,t_{j+1}), y\in [0,1]}\{|b_4(y,t)|\} $, and $\bar{B}_5=\max_{t\in (t_j,t_{j+1})} \{|B_5(t)|\}$.
	\subsection{Main Result}
	The flow diagram of the signals in the closed-loop system consisting of
	the plant, the controller, the parameter esimator, the neural
	operator and the ETM is presented in Fig. \ref{fig1}. The main
	results of the neural operator-based adaptive event-triggered control design are shown as follows.
	\begin{figure}
		\hspace{0.1cm}
		\includegraphics[width=8cm]{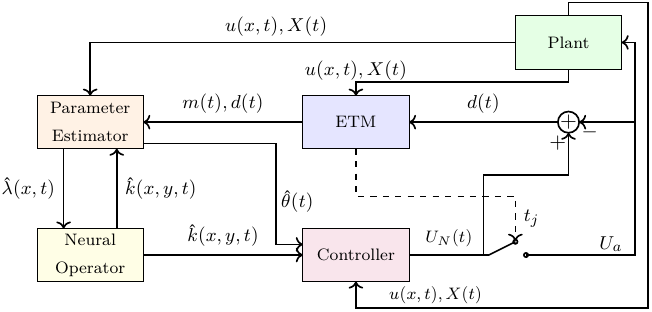} 	
		\caption{ The flow diagram of the signals in the closed-loop system under the \textcolor{black}{neural operator}-approximated adaptive event-triggered control.} 
		\label{fig1}
	\end{figure}

	\begin{theorem} \label{theorem2}
		For all initial conditions $u[0] \in L^2(0,1)$ and $X(0) \in \mathbb{R}^n$,	the closed-loop system, which consists of the plant \eqref{0}--\eqref{3} and
		the neural operator-based adaptive event-triggered control law \eqref{c162} that includes the event-triggering
		mechanism \eqref{c200}--\eqref{mtdefine10} where $d(t)$ is defined in \eqref{c1701}, the update law \eqref{eventAdaptivelawlambda}--\eqref{equation110}, and neural operator $\hat{\mathcal{K}}$, has the following properties:\\
		1) The estimated parameters $\hat{\lambda}(x,t), \hat{\theta}(t)$ and  neural kernel $\hat{k}(x,y,t)$ in closed-loop system are bounded for all $  t\in \mathbb{R}_+$  and $$\lim_{t\rightarrow \infty} \big(\|u[t]\|^2+|X(t)|^2+|d(t)|+m(t)\big)=0.$$ 	\\
		2) The neural operator-based adaptive event-triggered control signal is convergent to zero,
		i.e., $\lim_{t \rightarrow \infty} U_a(t) = 0$.\\
		3)  Additionally, there exist constants $\rho, R>0$ such that the stability estimate
		\begin{align}
			& \Gamma(t) \leq R\left(e^{\rho \Gamma(0)}-1\right)
		\end{align}
		where
		\begin{align}
			\Gamma(t) = & \|u[t]\|^2+ |X(t)|^2 + m(t) + f(t)d(t)^2\notag\\
			&+\|\tilde{\lambda}[t]\|^2 + \tilde{\theta}^T(t)\tilde{\theta}(t)  \label{Gamma}
		\end{align}
		holds for all $t \geq 0$.
	\end{theorem}
	\begin{proof}
		See the next details.
	\end{proof}
	Take the backstepping transformation
	\begin{align}
		\hat{w}(x, t)=u(x, t)-\int_0^x \hat{k}(x, y,t) u(y, t) d y-   \check{\gamma}(x,t) X(t).  \label{bretrans}
	\end{align}
	Applying the backstepping transformation \eqref{bretrans} into the original system \eqref{0}--\eqref{3},  we arrive at the following system:
	\begin{align}
		&  \dot{X}(t)  =    (A+BK) ( \hat{\theta}(t)) X(t) + B(\hat{\theta}(t)) \hat{w} (0,t) \notag\\
		& \qquad \quad+ A(\tilde{\theta}(t)) X(t) + B(\tilde{\theta}(t)) u (0,t), \\
		&  \hat{w}_t(x, t)=   \hat{w}_{x x}(x, t)-\check{\gamma}(x,t) B(\hat{\theta}(t)) \hat{w} (0,t) \notag\\
		&\qquad \qquad-\check{\gamma}(x,t) A(\tilde{\theta}(t))X(t)  - \check{\gamma}(x,t) B(\tilde{\theta}(t)) u (0,t)\notag\\
		& \qquad \qquad + (\tilde{\lambda}(x,t)) u(x,t) + 	\delta_{k 0}(x,t) u(x,t) \notag\\
		& \qquad \qquad - \int_{0}^{x} \hat{k}(x,y,t) \tilde{\lambda}(y,t) u(y,t) dy \notag\\	
		& \qquad \qquad  +\int_0^x \delta_{k 1}(x, y,t) u(y, t) d y \notag\\
		& \qquad \qquad -\int_0^x \hat{k}_t(x, y, t)  {u}(y, t) \mathrm{d} y - \check{\gamma}_t(x,t) X(t) , \label{1211} \\
		&  \hat{w}_x(0, t)=  \tilde{k}(0,0,t) \hat{w}(0,t) \label{1311}, \\
		&  \hat{w}_x(1, t)= U(t) -\wp \hat{w}(1,t)+ \tilde{k}(1,1,t) {u}(1,t)\notag\\
		& \qquad \qquad - (\wp \check{\gamma}(1,t)+\check{\gamma}_x(1,t) )X(t)\notag\\
		&\qquad \qquad -\int_{0}^{1}(\wp \hat{k}(1,y,t)+\hat{k}_x(1,y,t)) u(y,t) dy ,
	\end{align}
	where
	\begin{align}
		\delta_{k 0}(x,t)   =&-2 \frac{d}{d x}(\tilde{k}(x, x,t)), \\
		\delta_{k 1}(x, y,t)   =&-   \tilde{k}_{xx}(x, y,t)+    \tilde{k}_{yy}(x, y,t)+\hat{\lambda}(y,t) \tilde{k}(x, y,t),\\
		A(\tilde{\theta}(t)) = & \sum_{i=1}^{p} \tilde{\theta}_i (t) A_i, \\
		B(\tilde{\theta}(t)) = & \sum_{i=1}^{p} \tilde{\theta}_i (t) B_i.
	\end{align}
	Equations \eqref{supk} and \eqref{calk} imply that
	\begin{align}
		\|\delta_{k 0}\|_{\infty} \leq \iota, \label{delinf1}\\
		\|\delta_{k 1}\|_{\infty}  \leq \iota, \label{delinf2} \\
		\tilde{k}(x,y,t) \leq \iota. \label{delinf3}
	\end{align}
	Introducing the second transformation
	\begin{align}
		\hat{\beta}(x,t) = \hat{w}(x,t) - \int_0^x \check{h}(x,y) \hat{w}(y,t) dy \label{bwc}
	\end{align}
	with choosing the control input $U(t)$ as a neural operator-based adaptive event-triggered controller $U_a (t)$ given in \eqref{c162}, we derive the target system as
	\begin{align}
		\dot{X}(t) =&  ( A +B  K ) (\hat{\theta}(t)) X(t) + B(\hat{\theta}(t)) \hat{\beta} (0,t)\notag\\
		& + A(\tilde{\theta}(t)) X(t) +B(\tilde{\theta}(t)) u (0,t), \label{targ112} \\
		\hat{\beta}_t(x,t) = & \hat{\beta}_{xx} (x,t)     - \check{\gamma}(x,t) A(\tilde{\theta}(t)	) X( t) \notag \\
		&  - \check{\gamma}(x,t) B(\tilde{\theta}(t)	) u(0,t)  +  \tilde{\lambda}(x,t)  u(x,t)\notag\\
		& + \delta_{k 0} (x,t) u(x,t)+ \int_{0}^{x} \delta_{k 1}(x,y,t) u(y,t) dy \notag\\
		&  - \int_{0}^{x} \hat{k}(x,y,t)  \tilde{\lambda}(y,t) u(y,t) dy \notag\\
		& -\int_0^x \hat{k}_t(x, y, t)  {u}(y, t) \mathrm{d} y - \check{\gamma}_t(x,t) X(t) \notag  \\
		& - \int_{0}^{x} \check{h}_t(x,y,t) \hat{w}(y,t) dy, \label{hbeta1} \\
		\hat{\beta}_x(0,t) = & \tilde{k}(0,0,t) \hat{\beta}(0,t), \label{targ212}\\
		\hat{\beta}_x(1,t) = & -r\hat{\beta}(1,t)+ \tilde{k}(1,1,t) u(1,t)-d(t). \label{hbeta3}
	\end{align}
	Define the following constants for the bounds for both the   backstepping kernels and inverse backstepping kernels as
	\begin{align}
		& \|\hat{k}\|_{\infty}=\| \check{k}+\tilde{k} \|_{\infty} \leq 2 \bar{\lambda} e^{4 \bar{\lambda}}+\iota := \bar{k}, \label{bark}\\
		& \|\hat{k}^I\|_{\infty} \leq\|\hat{k}\|_{\infty} e^{\|\hat{k}\|_{\infty}} \leq \bar{k} e^{\bar{k}} := \bar{k^I} . \label{barl}
	\end{align}
	According to \cite{BHAN2025105968}, recalling \eqref{42q}, the following bound holds:
	\begin{align}
		\|\check{k}_t \|_{\infty} \leq & e^{4\bar{\lambda}} (1+2 \bar{\lambda} e^{4 \bar{\lambda} }) \| \hat{\lambda}_t  \| . \label{ckt}
	\end{align}
		Using \eqref{eventAdaptivelawlambda}, \eqref{evl}, we obtain
	\begin{align}
		\| \hat{\lambda}_t \| \leq  & \frac{r_1(1+\bar{k})}{r_a} \big(2(1+\bar{k^I})(1+\bar{h^I})+\bar{\gamma^I}\big) \notag\\
		& +\frac{r_1(1+\bar{k})\bar{\gamma^I}}{2r_b \underline{\lambda_{\min}} } . \label{lamt}
	\end{align}
	Applying the bound \eqref{lamt}, recalling \eqref{ckt} and \eqref{supk}, we then derive
	\begin{align} 
		\|\hat{k}_t\|_{\infty} \leq  \iota +   C_{k1}  \frac{r_1}{r_a} + C_{k2}  \frac{r_1}{r_b}   \label{hatkt}
	\end{align}
	where
	\begin{align}
		C_{k1}=& e^{4\bar{\lambda}} (1+\bar{k})^2 \big(2(1+\bar{k^I})(1+\bar{h^I})+\bar{\gamma^I}\big) ,\\
		C_{k2}=& e^{4\bar{\lambda}} (1+\bar{k})^2  \frac{  \bar{\gamma^I}}{2  \underline{\lambda_{\min}} }  .
	\end{align}
     are positive constants  independent of $t$ and design parameters.
	Recalling \eqref{eq:gamma} and \eqref{checkg}, there exists a positive constant $C_{\gamma}$ such that  
	\begin{align}
		\|\check{\gamma}_t\|_{\infty} \leq & C_{\gamma} \| \dot{\hat{\theta}} \|_{\infty}. \label{eq:gt}
	\end{align}
	\begin{lema}	\label{boundforkernel_t}
		The kernel time derivative $\check{h}_t$ exists and is uniformly bounded, i.e., for some positive constant  $C_h$
		\begin{align}
			{	\|\check{h}_t\|_{\infty} \leq} &  { C_{h} \| \dot{\hat{\theta}} \|_{\infty} }	\label{eq88} .
		\end{align}
	\end{lema}
	\begin{proof}
		See Appendix \ref{Appendix3} for details.  
	\end{proof}

	The inverse transformation $  \hat{\beta} \mapsto \hat{w} \mapsto  {u}$ is given in the form
	\begin{align}
		{u}(x, t)= & \hat{w}(x, t)+\int_0^x \hat{k}^I(x, y,t) \hat{w}(y, t) d y + \check{\gamma}^I (x,t) X(t), \label{xit} \\
		\hat{w}(x,t) = & \hat{\beta}(x,t) + \int_{0}^{x} \check{h}^I(x,y) \hat{\beta}(y,t) dy .\label{wbh}
	\end{align}
	It follows that
	\begin{align}
		u(x,t) = & \hat{\beta}(x,t) + \int_{0}^{x} \check{h}^I(x,y) \hat{\beta}(y,t) dy + \check{\gamma}^I(x,t) X(t) + \notag\\
		&  \int_{0}^{x} \hat{k}^I(x,y,t) \big(\hat{\beta}(y,t) + \int_{0}^{y} \check{h}^I(y,z) \hat{\beta}(z,t) dz \big) dy.\label{eq128} 
	\end{align}
	  For $\hat{\beta}(x,t)$-subsystem \eqref{hbeta1}--\eqref{hbeta3}, we obtain
		\begin{align}
			\frac{d}{dt}  \frac{1}{2} \|\hat{\beta}[t]\|^2 = &- \|\hat{\beta}_x[t]\|^2 - r \hat{\beta}^2(1,t)- \hat{\beta}(1,t)  d(t) \notag\\
			& + \hat{\beta}(1,t) \tilde{k}(1,1,t) u(1,t) - \tilde{k}(0,0,t) \hat{\beta}^2(0,t) \notag\\
			& + \int_{0}^{1} \hat{\beta}(x,t) \bigg( - \check{\gamma}(x,t) A(\tilde{\theta}(t)	) X( t) \notag \\
			&  - \check{\gamma}(x,t) B(\tilde{\theta}(t)	) u(0,t)+  \tilde{\lambda}(x,t)  u(x,t) \notag\\
			& + \delta_{k 0}(x,t) u(x,t) + \int_{0}^{x} \delta_{k 1}(x,y,t) u(y,t) dy  \notag\\ 
			&  - \int_{0}^{x} \hat{k}(x,y,t)  \tilde{\lambda}(y,t)  u(y,t) dy \notag\\
			& -\int_0^x \hat{k}_t(x, y, t)  {u}(y, t) \mathrm{d} y - \check{\gamma}_t(x,t) X(t) \notag  \\
			& - \int_{0}^{x} \check{h}_t(x,y,t) w(y,t) dy   \bigg) dx,  \label{ddtbeta}
		\end{align}	
		 For $X(t)$-subsystem,  for  a positive definite matrix $P(\hat{\theta})$   to satisfy \eqref{PQmatrix}, we obtain
		\begin{align}
			\frac{d}{dt} X^T(t) P(\hat{\theta} ) X(t) = & - X^T(t) Q(\hat{\theta}) X(t) +2 X^T(t) P(\hat{\theta}) \notag\\
			& B(\hat{\theta}) \hat{\beta}(0,t)  +2 X^T(t) P(\hat{\theta}) A(\tilde{\theta}) X(t)\notag\\
			& +2 X^T(t) P(\hat{\theta}) B(\tilde{\theta} ) u(0,t) \notag\\
			& +  \sum_{i=1}^p   \dot{\hat{\theta}}_i   X^T(t) \frac{\partial P(\hat{\theta})}{\hat{\theta}_i} X(t). \label{ddtX}
		\end{align}
 Define a Lyapunov function as
	\begin{align}
		V (t) =& \ln \big(1+ \frac{r_a}{2} \|\hat{\beta}[t]\|^2 +  {r_b}   X^T(t) P(\hat{\theta}) X(t)  + m(t) \notag\\
		&+   f(t) d^2(t) \big) + \frac{r_a}{2 r_1} \|\tilde{\lambda}[t]\|^2 + \frac{r_b}{2r_2} \tilde{\theta}^T(t) \tilde{\theta}(t) , \label{V2t}
	\end{align}
	which is constructed to capture the energy of the PDE state, the parameter estimation errors, and the effects of event-triggered mechanism. 
	
	Using  \eqref{c21}, \eqref{ddtbeta}, and \eqref{ddtX},   we can obtain for $t \in (t_j, t_{j+1})$
	\begin{align}
		\dot{V}  = &	 \frac{ r_a}{1+\Upsilon(t)  } \times \bigg( - \|\hat{\beta}_x[t]\|^2 - r \hat{\beta}^2(1,t)- \hat{\beta}(1,t)  d(t) \notag\\
		& + \hat{\beta}(1,t) \tilde{k}(1,1,t) u(1,t) - \tilde{k}(0,0,t) \hat{\beta}^2(0,t) \notag\\
		& + \int_{0}^{1} \hat{\beta}(x,t) \big( - \check{\gamma}(x,t) A(\tilde{\theta}(t)	) X( t) \notag \\
		&  - \check{\gamma}(x,t) B(\tilde{\theta}(t)	) u(0,t)   + \tilde{\lambda}(x,t) u(x,t) \notag\\
		&  - \int_{0}^{x} \hat{k}(x,y,t)   \tilde{\lambda}(y,t)  u(y,t) dy  \big) dx  \notag\\
		& + \Delta_1(t) + \Delta_2(t) - \Delta_3(t) - \Delta_4(t) - \Delta_5(t) \bigg) \notag\\
		& +  \frac{ r_b}{1+\Upsilon(t)  } \times \big(  - X^T(t) Q(\hat{\theta}) X(t)  \notag\\
		&+2 X^T(t) P(\hat{\theta}) B(\hat{\theta}) \hat{\beta}(0,t)  +2 X^T(t) P(\hat{\theta}) A(\tilde{\theta}) X(t)\notag\\
		& +2 X^T(t) P(\hat{\theta}) B(\tilde{\theta} ) u(0,t)  +  \sum_{i=1}^p   \dot{\hat{\theta}}_i   X^T \frac{\partial P(\hat{\theta})}{\hat{\theta}_i} X \big) \notag\\
		&+  \frac{ 1}{1+\Upsilon(t)  } \times \big( -\eta m(t) +\kappa_1\| {u}[t]\|^2 +2f(t)d(t)\dot{d}(t)   \notag\\
		&+ \dot{f}(t) d^2(t) \big)-\frac{r_a }{r_1}\int_{0}^{1} \tilde{\lambda}(y,t)  \textcolor{black}{\hat{\lambda}_t (y,t)} dy -\frac{r_b}{r_2}  \tilde{\theta}(t)^T \textcolor{black}{\dot{\hat{\theta}}(t)} ,
	\end{align}
	for $t \in (t_j, t_j +\tau)$, where 
	\begin{align}
		\Upsilon (t)=&\frac{r_a}{2} \|\hat{\beta}[t]\|^2 +  {r_b}   X^T(t) P(\hat{\theta}) X(t) + m(t)+f(t) d^2(t),\label{Upsilon} \\
		\Delta_1=& \int_{0}^{1} \hat{\beta} (x,t) \delta_{k 0} (x,t) u(x,t) dx , \\
		\Delta_2=& \int_{0}^{1} \hat{\beta} (x,t) \int_0^x \delta_{k 1}(x,y,t) u(y,t) dy dx, \\
		\Delta_3=&  \int_{0}^{1} \hat{\beta}(x,t)     \int_0^x \hat{k}_t(x, y, t)  {u}(y, t)  {d} y dx  , \\
		\Delta_4=&  \int_{0}^{1} \hat{\beta}(x,t) \check{\gamma}_t(x,t) X(t) dx ,  \\
		\Delta_5=&  \int_{0}^{1} \hat{\beta}(x,t)   \int_{0}^{x} \check{h}_t(x,y,t) w(y,t) dy  dx	.
	\end{align}
		Recalling \eqref{equation105}, \eqref{eventAdaptivelawtheta},  and \eqref{Upsilon}, there exist  positive constants  $\bar{\theta_t} ,C_{\theta 1},C_{\theta 2},C_{\theta 3}$   independent of $t$ and design parameters such that
	\begin{align}
		\|\dot{\hat{\theta}}\|_{\infty} \leq & \bar{\theta_t}  +C_{\theta 3} \frac{r_2+r_3}{\Upsilon(t)} u^2(0,t)  \label{suptheta}
	\end{align}
where
\begin{align}
 \bar{\theta_t} =	C_{\theta 1} \frac{r_2}{\min\{ r_a, r_b \}} +   C_{\theta 2} \frac{r_3}{\min\{ r_a, r_b \}}.  \label{barthetat}
\end{align}
	Defining  
	$  M_p = \max _{1 \leq i \leq p} \sup _{\hat{\theta} \in \Pi}\left\|\frac{\partial P(\hat{\theta})}{\partial \hat{\theta}_i}\right\| $ 
	and using \eqref{suptheta}, one obtains
	\begin{align}
		\left|
		\sum_{i=1}^p \dot{\hat{\theta}}_i \, X^T \frac{\partial P(\hat{\theta})}{\partial \hat{\theta}_i} X
		\right|
		\le p\, M_p\, \|\dot{\hat{\theta}}\|_{\infty}\, |X(t)|^2. \label{equation171}
	\end{align}
		Choosing $r_3$ in \eqref{eventAdaptivelawtheta} as  
	\begin{align}
		r_3= \frac{r_2 r_a}{r_b}, \label{r3}
	\end{align}
    substituting the expressions of \eqref{eventAdaptivelawlambda}--\eqref{equation110} and \eqref{f(t)},  we obtain for $t \in (t_j, t_j +\tau)$
	\begin{align}
		\dot{V}  \leq &   \frac{ r_a}{1+\Upsilon (t) } \times \big( - \|\hat{\beta}_x[t]\|^2 - r \hat{\beta}^2(1,t)- \hat{\beta}(1,t)  d(t) \notag\\
		& + \hat{\beta}(1,t) \tilde{k}(1,1,t) u(1,t) - \tilde{k}(0,0,t) \hat{\beta}^2(0,t) \notag\\
		& + \Delta_1(t) + \Delta_2(t) - \Delta_3(t) - \Delta_4(t) - \Delta_5(t)  \big) \notag\\
		& +  \frac{ r_b}{1+\Upsilon (t) } \times \big(  - X^T(t) Q(\hat{\theta})  X(t) \notag\\
		& +  \sum_{i=1}^p   \dot{\hat{\theta}}_i   X^T \frac{\partial P(\hat{\theta})}{\hat{\theta}_i} X+2 X^T(t) P(\hat{\theta}) B(\hat{\theta}(t)) \hat{\beta} (0,t)  \big)  \notag\\
		&+  \frac{ 1}{1+\Upsilon(t)  } \times \big( -\eta m(t) +\kappa_1\| {u}[t]\|^2  +\frac{\dot{d}(t)^2}{a_2} \notag\\
		& -(a_0 + a_1 f(t)) d^2(t) \big)   .
	\end{align}
			Using \eqref{delinf1}--\eqref{delinf3},  \eqref{hatkt}--\eqref{eq128}, and \eqref{suptheta}, one derive that for some positive constants $C_3,C_4,C_5 $ independent of $t$ and design parameters,
	\begin{align}
		\Delta_1  \leq & \iota   \|\hat{\beta}\|^2 (1+\bar{h}+ \bar{k^I} + \bar{h} \bar{k^I}  )+\iota  \bar{\gamma^I}  \| \hat{\beta}\| |X(t)|, \label{delta1} \\
		\Delta_2  \leq & \iota   \|\hat{\beta}\|^2 (1+\bar{h}+ \bar{k^I} + \bar{h} \bar{k^I}  )+\iota  \bar{\gamma^I}  \| \hat{\beta}\| |X(t)|, \label{delta2} \\
		\Delta_3  \leq &  \bigg(\iota+    C_{k1}  \frac{r_1}{r_a} + C_{k2}  \frac{r_1}{r_b}  \bigg)   \big(     \|\hat{\beta}\|^2 (1+\bar{h}+ \bar{k^I} + \bar{h} \bar{k^I}  ) \notag\\
		&+  \bar{\gamma^I}  \| \hat{\beta}\| |X(t)|\big)  \notag\\
		\leq & \iota   \|\hat{\beta}\|^2 (1+\bar{h}+ \bar{k^I} + \bar{h} \bar{k^I}  )+\iota  \bar{\gamma^I}  \| \hat{\beta}\| |X(t)|\notag\\
		& + \frac{r_1}{\min\{ r_a, r_b \}} C_3 (\|\hat{\beta}\|^2 + |X(t)|^2) ,\\
		\Delta_4 \leq  &  { C_{\gamma} }	\|\dot{\hat{\theta}}\|_{\infty}    \| \hat{\beta}[t]\|  |X(t)| \notag\\
		 \leq &  C_4  \frac{r_2+r_3}{\min\{ r_a, r_b \}} (\|\hat{\beta}\|^2 + |X(t)|^2+u^2(0,t)) ,\\
		\Delta_5 \leq  &  { C_{h} } 	\|\dot{\hat{\theta}}\|_{\infty}   \| \hat{\beta}[t]\|   \|\hat{w}[t]\|   \notag\\
		\leq &  C_5 \frac{r_2+r_3}{\min\{ r_a, r_b \}} ( \| \hat{\beta}[t]\|^2  + u^2(0,t)). \label{Delta5}
	\end{align}
	One obtains from \eqref{xit} and \eqref{wbh} that
	\begin{align}
		u(0,t) = & \hat{\beta}(0,t) + \check{\gamma}^I(0,t) X(t)   ,\\
		u(1,t) = & \hat{\beta}(1,t) + \int_{0}^{1} \check{h}^I(1,y) \hat{\beta}(y,t) dy + \check{\gamma}^I(1,t) X(t) + \notag\\
		&  \int_{0}^{1} \hat{k}^I(1,y,t) \big(\hat{\beta}(y,t) + \int_{0}^{y} \check{h}^I(y,z) \hat{\beta}(z,t) dz \big) dy.
	\end{align}
	It can be easily derived that \eqref{eq148}--\eqref{eq74}  still holds true replacing $\beta$ with $\hat{\beta}$, i.e.,
	\begin{align}
		\hat{\beta}^2(0,t) \leq&   \hat{\beta}^2(1,t)+\|  \hat{\beta}_x[t]\|^2+\|  \hat{\beta}[t]\|^2, \label{eq1481}\\
		u^2(0,t) \leq  & 2  ( \hat{\beta}^2(1,t)+\|  \hat{\beta}_x[t]\|^2+\|  \hat{\beta}[t]\|^2 +   \bar{\gamma^I}^{2}  |X(t)|^2),   \label{eq1451} \\
		u^2(1,t) \leq & 6  \hat{\beta}^2(1,t)+  6 \bar{h^I}^2  \| \hat{\beta}[t]\|^2 + 3 \bar{k^I}^2  (1+\bar{h^I})^2 \| \hat{\beta}[t]\|^2  \notag\\
		& + 3  \bar{\gamma^I}^{2}  |X(t)|^2 , \label{eq1441} \\
		\|u[t]\|^2 \leq &   2 (1+\bar{k^I})^2 (1+\bar{h^I})^2 \| \hat{\beta}[t]\|^2 + 2 \bar{\gamma^I}^{2} |X(t)|^2  , \label{equation701} \\
		-\| \hat{\beta}_x[t]\|^2  \leq & \frac{1}{2}  \hat{\beta}^2(1,t) - \frac{1}{4} \| \hat{\beta}[t]\|^2. \label{eq741}
	\end{align}
	Using \eqref{ddt}, and  \eqref{equation171}--\eqref{eq741}, we obtain that for some positive constans $l_1, l_2, l_3, l_4$,
	\begin{align}
		\dot{V} \leq& \frac{1}{1+\Upsilon(t)} \times \bigg(  -\nu_1 \| \hat{\beta}[t]\|^2- \nu_2   \|\hat{\beta}_x[t]\|^2-\nu_3 \hat{\beta}^2(1,t) \notag\\
		& -\nu_4 |X(t)|^2 -\nu_5 d(t)^2 -\eta m(t) \bigg) \label{DotV}
	\end{align}
		for $t \in\left(t_j, t_{j}+\tau \right)$,where
	\begin{align}
		\nu_1 = &\frac{r_a}{8}-  \frac{2\epsilon_2}{a_2} - (3 \bar{k^I}^2  (1+\bar{h^I})^2 +  6 \bar{h^I}^2)  \frac{\epsilon_3}{a_2} -\frac{r_b}{4 \delta_2}-\iota l_1 \notag\\
		&- 2 (1+\bar{k^I})^2 (1+\bar{h^I})^2 (\kappa_1 +  \frac{\epsilon_4}{a_2})- \frac{r_1 r_a}{\min\{ r_a, r_b \}} C_3 \notag\\
		&-\frac{(r_2+r_3)r_a}{\min\{ r_a, r_b \}}   C_4- \frac{3(r_2+r_3)r_a}{\min\{ r_a, r_b \}}   C_5,  \label{nu1}\\
		\nu_2 = &\frac{r_a}{2}- \frac{2 \epsilon_2}{a_2}-\frac{r_b}{4 \delta_2}-\iota l_2- \frac{2(r_2+r_3)r_a}{\min\{ r_a, r_b \}}   C_5,  \\
		\nu_3 = &  {(r-\frac{1}{4}-\delta_1)} r_a - \frac{2\epsilon_2}{a_2}   -   \frac{6 \epsilon_3}{a_2} -\frac{r_b}{4 \delta_2}-\iota l_3\notag\\
		&- \frac{2(r_2+r_3)r_a}{\min\{ r_a, r_b \}}   C_5, \\
		\nu_4 = & r_b \lambda_{\min} (Q) -2 r_b \delta_2 \lambda_{\max} (PBB^TP^T) -  \frac{ 3  \bar{\gamma^I}^{2}  \epsilon_3}{a_2}   \notag\\
		&  -2 \bar{\gamma^I}^{2} (  \frac{\epsilon_2}{a_2} +\kappa_1 +  \frac{\epsilon_4}{a_2})-\iota l_4 -\frac{r_1 r_a}{\min\{ r_a, r_b \}} C_3\notag\\
		&-\frac{(r_2+r_3)r_a}{\min\{ r_a, r_b \}}   C_4-r_b p M_p \bar{\theta_t}- \frac{2(r_2+r_3)r_a}{\min\{ r_a, r_b \}}   C_5 , \label{nu4}\\
		\nu_5 = & a_1 f(t) +a_0-\frac{\epsilon_1}{a_2}-\frac{r_a  }{4 \delta_1}. \label{nu5}
	\end{align}
	Let $r_e = \frac{r_b}{r_a}$. Recalling \eqref{suptheta}, \eqref{barthetat}, and \eqref{r3}, we choose
	\begin{align}
		r_2 < & \frac{\underline{\lambda_{\min}} r_b \min\{ r_a, r_b \}  }{p M_p(r_b C_{\theta 1}+ {r_a} C_{\theta 2})} , \\
		r_1 < &  \frac{\frac{\min\{r_a,r_b\}}{8}-(r_2+r_3)(C_4+3C_5) }{C_3},\\
		\delta_2 < & \frac{\lambda_{\min}(Q)- p M_p \bar{\theta_t} }{2\lambda_{\max}(PBB^TP^T)} , \\
		r_b > & \frac{1}{\lambda_{\min}(Q)- 2\delta_2 \lambda_{\max} (PBB^TP^T)- p M_p \bar{\theta_t} } \times \notag\\
		&\bigg(   \frac{3  \bar{\gamma^I}^{2}   \epsilon_3}{a_2} +\frac{r_1}{\min\{ 1, r_e \}} C_3  +\frac{r_2+r_3}{\min\{1, r_e \}}   C_4   \notag\\
		& +\frac{2(r_2+r_3)}{\min\{1, r_e \}}   C_5  +2 \bar{\gamma^I}^{2} (  \frac{\epsilon_2}{a_2} +\kappa_1 +  \frac{\epsilon_4}{a_2}) \bigg),\\
		r_a > &    \frac{16 \epsilon_2}{a_2}  +\frac{2 r_b}{ \delta_2}+8(\frac{r_1}{\min\{ 1, r_e \}} C_3  +\frac{r_2+r_3}{\min\{ 1, r_e \}}   C_4  \notag\\
		&+\frac{3(r_2+r_3)}{\min\{ 1, r_e \}}   C_5 ) +    \frac{8 \epsilon_3}{a_2}(3 \bar{k^I}^2  (1+\bar{h^I})^2 +  6 \bar{h^I}^2)   \notag\\
		&+16 (1+\bar{k^I})^2 (1+\bar{h^I})^2 (\kappa_1 +  \frac{\epsilon_4}{a_2}) , \\
		\delta_1 < & \frac{(r-\frac{1}{4})r_a-   \frac{2 \epsilon_2}{a_2}   -  \frac{6 \epsilon_3}{a_2}-\frac{r_b}{4 \delta_2}- \frac{2(r_2+r_3) }{\min\{ 1, r_e \}}   C_5}{r_a}  ,\\
		a_0 \geq &  \frac{\epsilon_1}{a_2}+\frac{r_a  }{4 \delta_1} ,\\
		a_1 \geq & \eta ,
	\end{align}
	where $\underline{\lambda_{\min}}$ is defined in Assumption \ref{Assumption3}. Then, $\nu_1, \nu_2, \nu_3, \nu_4, \nu_5$ are positive   and
	\begin{align}
		\dot{V} \leq& \frac{1}{1+\Upsilon(t)} \times \bigg(  -\nu_1 \| \hat{\beta}[t]\|^2- \nu_2   \|\hat{\beta}_x[t]\|^2-\nu_3 \hat{\beta}^2(1,t) \notag\\
		& -\nu_4 |X(t)|^2 -\eta \big(f(t)d(t)^2+m(t) \big)\bigg) \label{equation190}
	\end{align}
	for $t \in\left(t_j, t_{j}+\tau\right)$.
	
	Similar to \eqref{DotV}, recalling \eqref{c211} and \eqref{f(t)}, choosing $	\lambda_d \geq  a_2 \omega_1^2+ \frac{\varepsilon_1}{a_2} +\frac{r_a }{4 \delta_0} + \eta  \omega_1^2$, we derive that \eqref{equation190} holds true for $t\in (t_j+\tau,t_{j+1})$.  Recalling \eqref{V2t}, the result \eqref{equation190} holds true for  $t \in\left(t_j, t_{j}+\tau\right)$  and for $t\in (t_j+\tau,t_{j+1})$  yields that $\|\hat{\beta}[t]\|^2, |X(t)|^2, m(t)+f(t)d(t)^2$  are bounded and integrable. Meanwhile, $\|\tilde{\lambda}[t]\|^2, \tilde{\theta}^T(t) \tilde{\theta}(t)$ are bounded.    Since $\omega_0 \leq f(t) \leq \omega_1$ and $m(t)>0$, we conclude that both $m(t)$ and $d(t)$ remain bounded.
	
	Taking the time derivative of $\frac{1}{2}\|\hat{\beta}[t]\|^2$, one obtains \eqref{ddtbeta}.  Using \eqref{ddt},    \eqref{supk}, \eqref{delta1}--\eqref{eq741}, and the fact that $\|\hat{\beta}[t]\|^2, |X(t)|^2, m(t)+f(t)d(t)^2, \|\tilde{\lambda}[t]\|^2, \tilde{\theta}^T(t) \tilde{\theta}(t)$  are all bounded, we conclude that $\frac{d}{dt}\frac{1}{2}\|\hat{\beta}[t]\|^2$ is bounded. Similarly, recalling \eqref{ddtX}, \eqref{c21}, \eqref{c211} and \eqref{f(t)}, we derive that  $\frac{d}{dt} |X(t)|^2$ and $ \frac{d}{dt}\big( m(t)+f(t)d(t)^2\big) $ are   bounded. Thus using Barbalat's lemma, we get $\|\hat{\beta}[t]\|^2$, $|X(t)|^2$, $m(t)$ and $d(t)$ tend to zero as $t \rightarrow \infty$.  This completes the first property of Theorem \ref{theorem2}. Recalling \eqref{c162}, we can straightforwardly obtain that the proposed neural operator-based adaptive event-triggered control input $U_a$ converges to zero as well. The second property of Theorem \ref{theorem2} is thus obtained.

	For the stability estimate, note that the Lyapunov function \eqref{V2t} satisfies
	\begin{align}
		\|\hat{\beta}[t]\|^2 \leq & \frac{2}{r_a} \left( e^{ V(t)}-1\right),\label{eq199} \\
		|X(t)|^2 \leq & \frac{1}{r_b \underline{\lambda_{\min}}}  \left( e^{ V(t)}-1\right),\\
		m(t)+f(t)d(t)^2 \leq &     e^{ V(t)}-1 , \\
		\|\tilde{\lambda}[t]\|^2   \leq & \frac{2r_1}{r_a} V(t) \leq \frac{2r_1}{r_a}  \left(e^{ V(t)}-1\right),\\
		\tilde{\theta}^T(t)\tilde{\theta}(t) \leq & \frac{2r_2}{r_b} V(t) \leq \frac{2r_2}{r_b}  \left(e^{ V(t)}-1\right) \label{eq203}
	\end{align} 
	for all $t \geq 0$. Further, using \eqref{equation701} and \eqref{eq199}--\eqref{eq203}, we obtain the following results
	\begin{align}
		\Gamma(t) \leq &   M_2 \left(e^{ V(t)}-1\right), \\
		V(t) \leq & M_3 \Gamma(t),
	\end{align}
	where
	\begin{align}
		M_2 = &\max \left \{  \frac{4 (1+\bar{k^I})^2 (1+\bar{h^I})^2 }{r_a} , \frac{2 \bar{\gamma^I}^{2}+1}{r_b \underline{\lambda_{\min}}},\frac{2r_1}{r_a} , \frac{2r_2}{r_b}   \right\} , \\
		M_3 = &\max \bigg \{   { r_a (1+\bar{k})^2 (1+\bar{h})^2 } , \frac{r_a}{2r_1} , \frac{r_b}{2r_2}, \notag\\
		&  \qquad \ \  r_b  (2 \bar{\gamma }^{2}+1){ \overline{\lambda_{\max}}}  \bigg \},
	\end{align} 
	where $\overline{\lambda_{\max}}$ is defined in Assumption \ref{Assumption3}, resulting in the  stability estimate
	\begin{align}
		\Gamma(t) & \leq M_2 \left(e^{M_3 \Gamma(0)}-1\right) .
	\end{align}
	This completes the proof of Theorem \ref{theorem2}.
	 
     In the implementation of a \textcolor{black}{neural operator}-approximated adaptive event-triggered controller for the reaction-diffusion PDE-ODE cascade system,  the design steps and results are shown in  Fig. \ref{fig12}.
	 	\begin{figure}
	 	\hspace{0.1cm}
	 	\includegraphics[width=8cm]{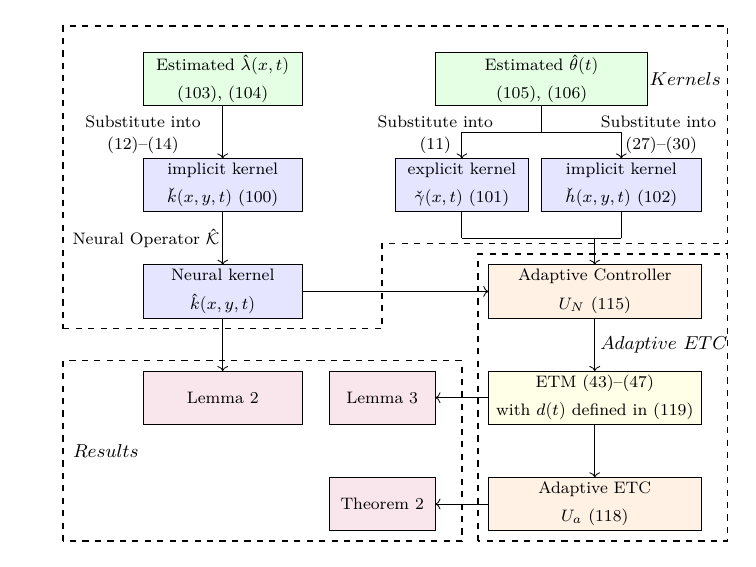} 	
	 	\caption{The design process and results under neural operator-based adaptive event-triggered control. } 
	 	\label{fig12}
	 \end{figure}
	 
     The following result shows that the neural operator is introduced
 primarily for computational efficiency, while the closed-loop
 stability can still be guaranteed when the exact kernel function
 is directly utilized.
 \begin{theorem}[\textcolor{black}{Neural operator}-free Stability Result]\label{theorem3}
 	Consider the closed-loop system consisting of the plant
 	\eqref{0}--\eqref{3}, the adaptive event-triggered controller,
 	and the update laws
 	\eqref{eventAdaptivelawlambda}--\eqref{equation110}.
 	Suppose that the neural operator approximation is not employed,
 	and the exact kernel function $\check{k}(x,y,t)$ is directly used
 	in place of the neural kernel $\hat{k}(x,y,t)$
 	in \eqref{c162} and \eqref{c1701}.
 	
 	Then, all conclusions of Theorem~\ref{theorem2} remain valid.
 	In particular:
 	
 	\begin{enumerate}
 		\item The estimated parameters
 		$\hat{\lambda}(x,t)$ and $\hat{\theta}(t)$
 		remain bounded for all $t\in\mathbb{R}_+$, and
 		$$
 		\lim_{t\rightarrow\infty}	\left(\|u[t]\|^2+|X(t)|^2+|d(t)|+m(t)\right)=0.
 		$$
 		
 		\item The adaptive event-triggered control input converges to zero.
 		
 		\item There exist constants $\rho_2,R_2>0$ such that
 		\begin{align}
 				\Gamma(t)\leq  	R_2\left(e^{\rho_2\Gamma(0)}-1\right), 			
 		\end{align} 	
 		where $\Gamma(t)$ is defined in \eqref{Gamma}.
 	\end{enumerate}
 \end{theorem}
 
 \begin{proof}
 	The proof follows similarly to that of
 	Theorem~\ref{theorem2}.	When the exact kernel function $\check{k}(x,y,t)$ is directly employed,
 	the neural approximation error associated with the neural operator
 	$\hat{k}(x,y,t)$ is eliminated.
 	Therefore, all Lyapunov estimates and stability arguments derived in
 	Theorem~\ref{theorem2} remain valid without modification. 	In particular, the closed-loop Lyapunov functional still satisfies
 	the same differential inequalities as in the proof of
 	Theorem~\ref{theorem2}, which guarantees boundedness of all adaptive
 	parameters and asymptotic convergence of the closed-loop system states in the sense of $L^2$ norm.
 	Furthermore, the same argument establishes the stability estimate for
 	$\Gamma(t)$. 	Hence, all conclusions stated above hold.
 \end{proof}
	
	\section{Numerical Simulation} \label{sec6}
	\subsection{Model and Design Parameters}

	In the simulation example, we consider the reaction--diffusion PDE with $\lambda(x)=10\cos(9\cos^{-1}x)+10, \qquad q=15,$	under the initial condition $u(x,0)=20.$ The ODE subsystem is chosen as
	 \[
   \begin{aligned}
       & A(\theta)=A_0+\theta A_1
		=
		\left(
		\begin{array}{cc}
			0 & 1 \\
			0 & 0
		\end{array}
		\right)
		+
		\theta
		\left(
		\begin{array}{cc}
			1 & 0 \\
			0 & 0
		\end{array}
		\right), \\
		& B=B_0=
		\binom{0}{1},
   \end{aligned}     
     \]
		with initial condition $X(0)=(x_1(0),x_2(0))^T=(1,1)^T.$
        
        Using the backstepping design, we construct the triple $(K,P,Q)$ as
	\[
	\begin{aligned}
		& K(\theta)
		=
		-\left(
		1+(\theta+1)^2 \quad \theta+2
		\right), \\
		& P(\theta)
		=
		\frac{1}{2}Q(\theta)
		=
		\left(
		\begin{array}{cc}
			1+(1+\theta)^2 & 1+\theta \\
			1+\theta & 1
		\end{array}
		\right),
	\end{aligned}
	\]
	which satisfies the Lyapunov equation \eqref{PQmatrix}.
	
	The numerical simulations are carried out using a finite-difference scheme. The spatial and temporal discretization steps are chosen as
	$
	\Delta x = 0.02, \ \Delta t = 1\times10^{-4}.
	$
	
	The parameters in the event-triggered mechanism, which are used for the nominal ETC, adaptive ETC, and neural operator-based adaptive ETC, are selected as $\eta=9.775,
	\ 
	\kappa_1=2240,
	\ 
	\lambda_d=375,
	\ 
	\omega_0=0.8.$
	The minimum dwell time enforced in the ETM is chosen as
	$
	\tau = 5\Delta t.
	$ 
		  To estimate the unknown parameters $\lambda(x)$ and $\theta$, the adaptation gains in \eqref{eventAdaptivelawlambda}--\eqref{eventAdaptivelawtheta} are selected as
		  $ r_1=900,
		 \ 
		 r_2=90,
		 \ 
		 r_3=225,
		 \ 
		 r_a=500,
		 \ 
		 r_b=200,$
		 and the projection bounds are chosen as
		  $\bar{\lambda}=30,
		 \ 
		 \bar{\theta}=2.$

		 \subsection{Simulation Results}
		 
	  It is easy to verify that the system is open-loop unstable.  Applying the proposed nominal ETC law \eqref{c16},    Fig.~\ref{figETCCTC} demonstrates that both the PDE state $u(x,t)$ and the ODE state $X(t)$ converge to zero. For clearer visualization, the trajectories of the ETC signal, the continuous-in-time signal, and the dynamic variable $m(t)$ over the interval from $10\,\mathrm{s}$ to $11\,\mathrm{s}$ are presented in Fig.~\ref{fignominalETC_control_m_zoom}.  The minimum triggering interval observed in the simulation is $ 8.1\times10^{-3}\,\mathrm{s}, $  while the average triggering interval is	 $ 5.79\times10^{-2}\,\mathrm{s}. $

    Applying the adaptive ETC law without neural operator approximation, Figs.~\ref{figAdaptiveETCCTC} shows that both the PDE state $u(x,t)$ and the ODE state $X(t)$ converge to zero.   The ETC signal, continuous-in-time signal, and dynamic variable $m(t)$ under adaptive ETC are shown in  Fig. \ref{figadaptiveETC_control_m_zoom}.       The minimum triggering interval observed in the simulation is
   $
   4.4\times10^{-3}\,\mathrm{s},
   $
   and the average triggering interval is
   $
   3.44\times10^{-2}\,\mathrm{s}.
   $  The estimated parameters $\hat{\lambda}(x,t)$ and $\hat{\theta}(t)$ are shown in Figs.~\ref{figadaptiveETClambda} and \ref{figadaptiveETCtheta}.
   
   To construct a neural operator approximation for the gain kernel $\check{k}(x,y,t)$ associated with unknown parameter estimates $\hat{\lambda}(x,t)$, a dataset representing parameter trajectories likely to arise during closed-loop operation must first be generated.  Specifically, we aim to learn the operator mapping
   $
   \hat{\mathcal{K}}:
   \hat{\lambda}(x,t)
   \mapsto
   \hat{k}(x,y,t).
   $   
   To generate the training data, we sample 10 different functions of the form
   $
   \lambda(x)
   =
   10\cos(\sigma\cos^{-1}x)+10,
   $
   where $\sigma$ is uniformly sampled from the interval $(8.5,9.5)$. For each sampled plant, the adaptive ETC closed-loop system is simulated using the finite-difference kernel solver.  Pairs of the form
   $
   (\hat{\lambda}(x,t),\check{k}(x,y,t))
   $
   are sampled every 100 time steps over the first 30 seconds of simulation, yielding a total of $30000$  $(\hat{\lambda}(x,t), \check{k}(x,y,t))$  pairs for training and testing.
   
   The entire dataset generation process requires more than four hours on an NVIDIA A6000 GPU for only 10 plants, which further motivates the use of neural operator approximations for real-time kernel evaluation.   For the neural operator architecture, we employ a DeepONet model in which the branch network is implemented as a convolutional neural network and the trunk network is implemented as a fully connected feedforward neural network.
   
    As shown in Figs. \ref{figDeepAdaptiveETCCTC} and  \ref{figDeepadaptiveETC_control_m_zoom}, the controller constructed using the DeepONet-approximated kernels achieves stabilization performance comparable to that obtained using the exact finite-difference kernels. This indicates that the learned kernels provide sufficient accuracy for closed-loop implementation.  The minimum triggering interval observed in the simulation is
  $   5.3\times10^{-3}\,\mathrm{s}, $  while the average triggering interval is  $
   3.55\times10^{-2}\,\mathrm{s}.
   $    
   The estimated parameters $\hat{\lambda}(x,t)$ and $\hat{\theta}(t)$ are shown in Figs.~\ref{figDeepadaptiveETClambda} and \ref{figDeepadaptiveETCtheta}.  The trajectories of $\hat{k}(1,y,t)$ and the kernel approximation error $\tilde{k}(1,y,t)$ are shown in Fig.~\ref{figDeepadaptiveETCk}.
   
   In this simulation, the average computation time for 100 evaluations of the kernel $k$ is $2.19\,\mathrm{ms}$ using the finite-difference method and $0.43\,\mathrm{ms}$ using the neural operator approximation, corresponding to a speedup of approximately $5.09\times$.  Computation times for different spatial discretization resolutions are summarized in Table~\ref{table_kernel_time}.

	\begin{figure}[!t]
		\centering
		\subfloat[Results for $ \|u\|_{L^2}+ |X(t)|^2$]{
			\includegraphics [width=4cm] {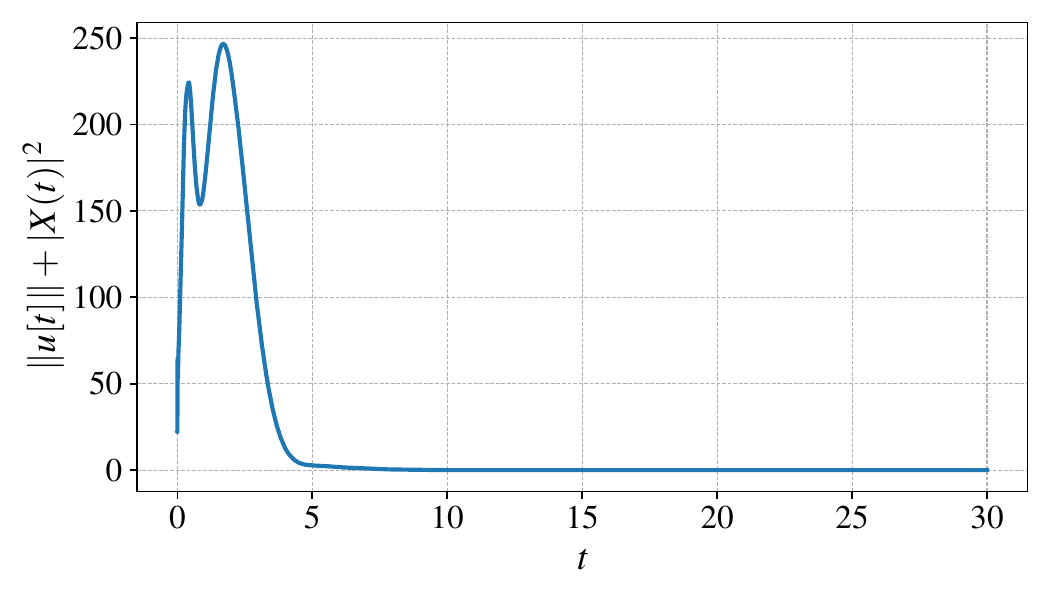}
			\label{figETCCTC}	
		}
		\subfloat[Trajectories involved in ETC signal, CTC signal and dynamic variable $m(t)$ from $10s$ to $11s$.]{
			\includegraphics [width=4cm] {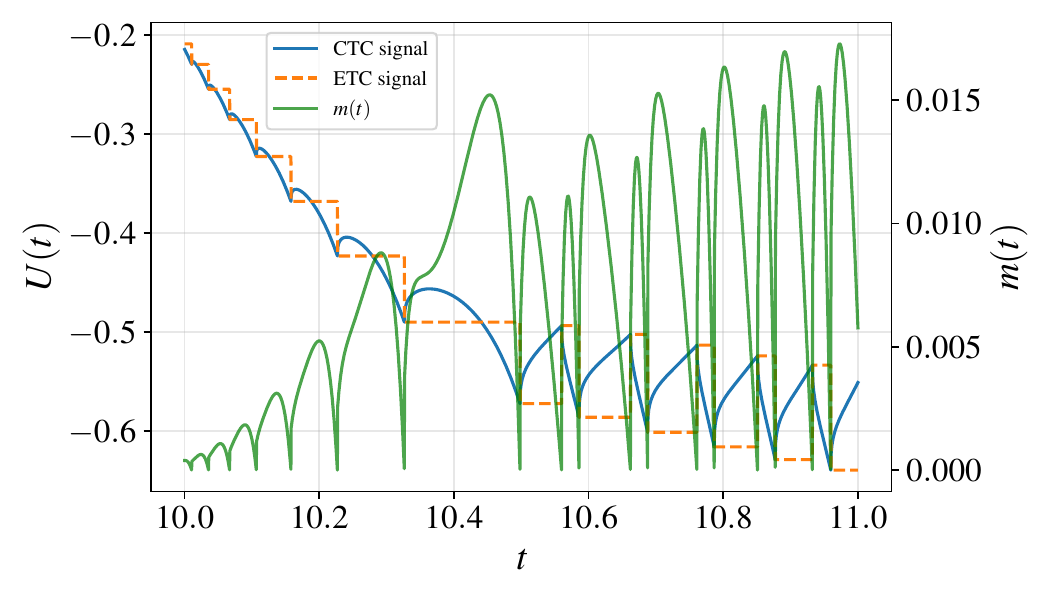}
			\label{fignominalETC_control_m_zoom}	
		} \\
		\caption{Results in closed-loop system with nominal ETC.} 
		\subfloat[Results for $ \|u\|_{L^2}+ |X(t)|^2$]{
			\includegraphics [width=4cm] {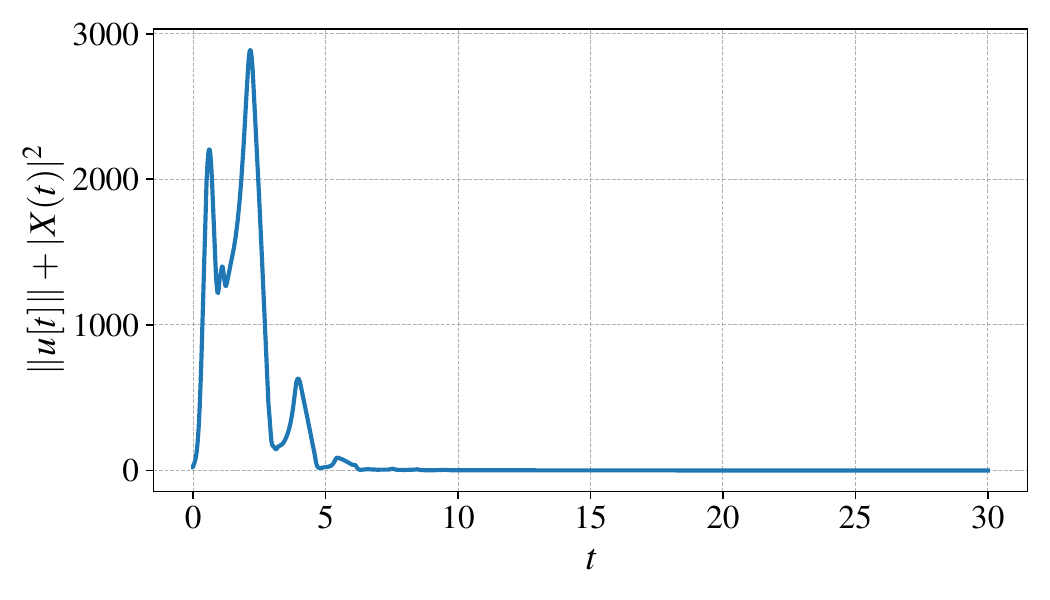}
			\label{figAdaptiveETCCTC}	
		}
		\subfloat[Comparison of ETC signal, CTC signal and dynamic variable $m(t)$ from $10s$ to $11s$.]{
			\includegraphics [width=4cm] {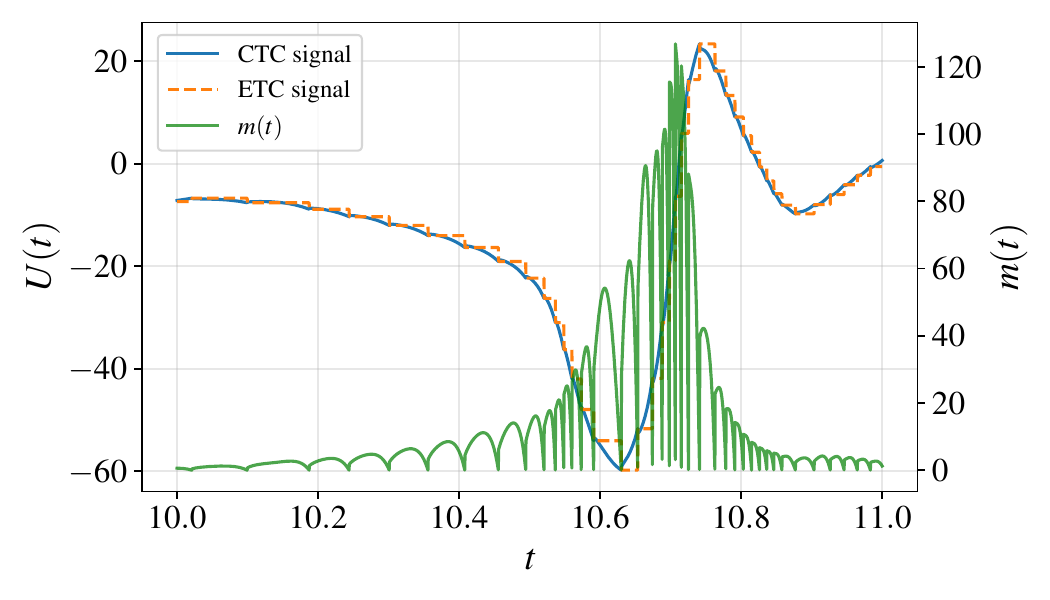}
			\label{figadaptiveETC_control_m_zoom}	
		} \\
		\caption{Results in closed-loop system with adaptive  ETC.}
		\subfloat[Results for $ \|u\|_{L^2}+ |X(t)|^2$]{
			\includegraphics [width=4cm] {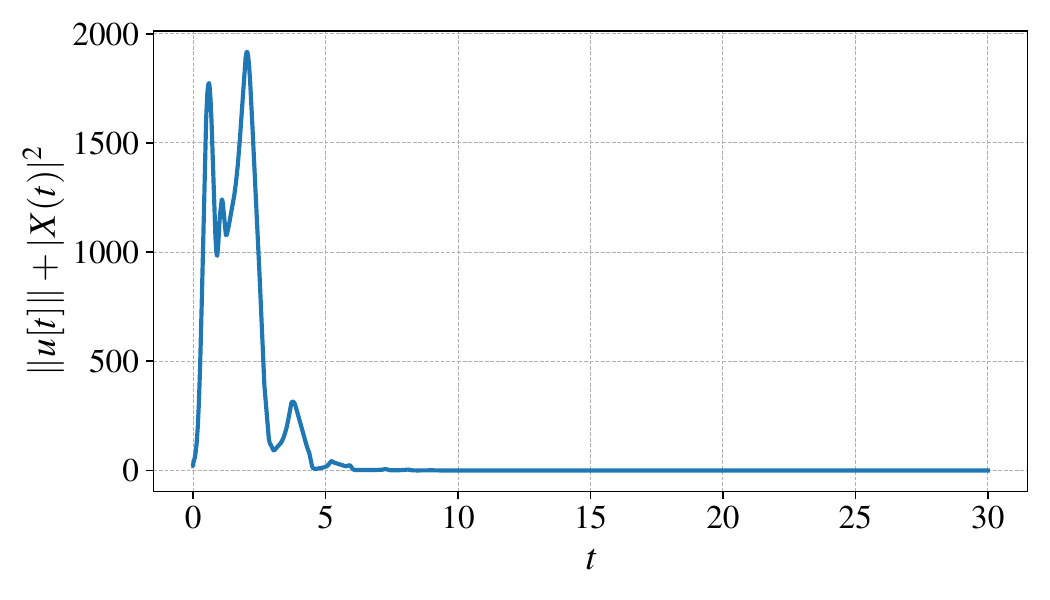}
			\label{figDeepAdaptiveETCCTC}	
		}
		\subfloat[Comparison of ETC signal, CTC signal and dynamic variable $m(t)$ from $10s$ to $11s$.]{
			\includegraphics [width=4cm] {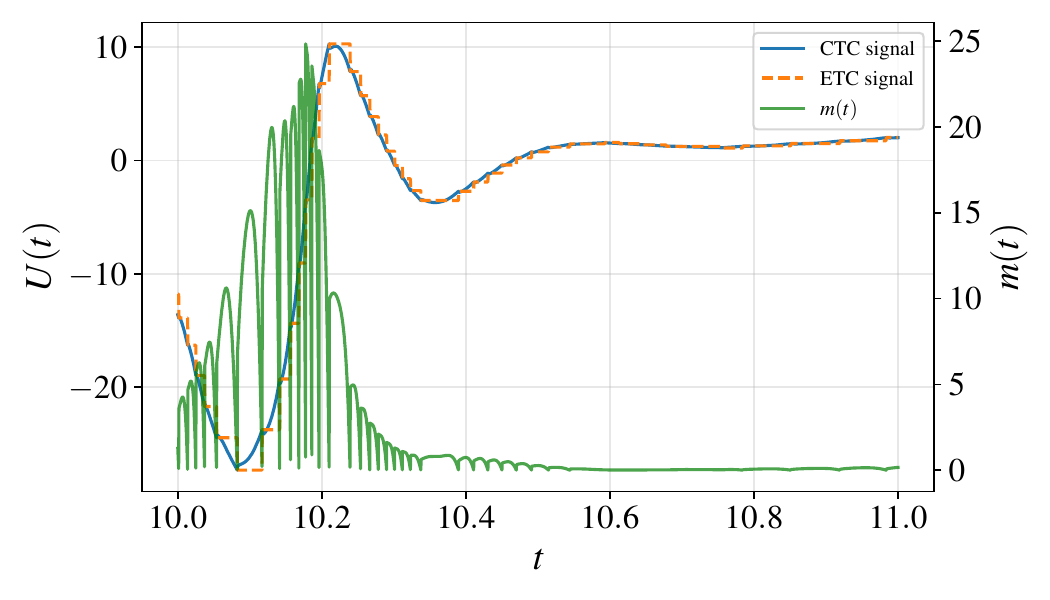}
			\label{figDeepadaptiveETC_control_m_zoom}	
		}
		\caption{Results in closed-loop system with neural operator-based adaptive  ETC.}
	\end{figure}

	\begin{figure}[!t]
		\centering
		\subfloat[Results for $\hat{\lambda}(x,t)$ estimate in closed-loop system with adaptive ETC.]{
			\includegraphics [width=4.1cm] {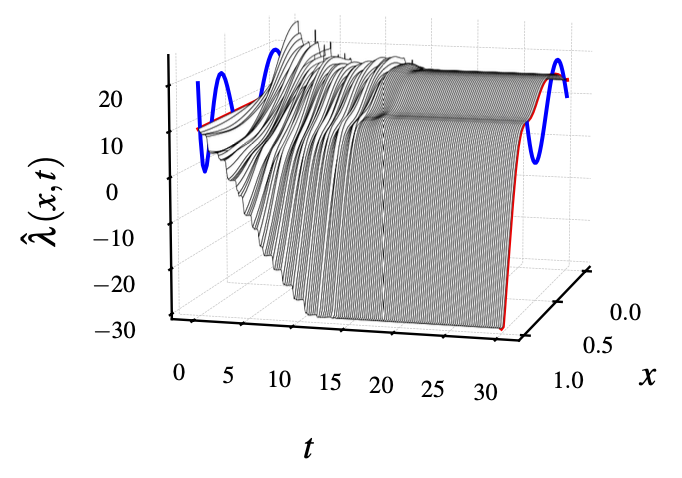}
			\label{figadaptiveETClambda}
		}	
	\subfloat[Results for $\hat{\lambda}(x,t)$ estimate in closed-loop system with  neural operator-based   adaptive ETC.]{
		\includegraphics [width=4.1cm] {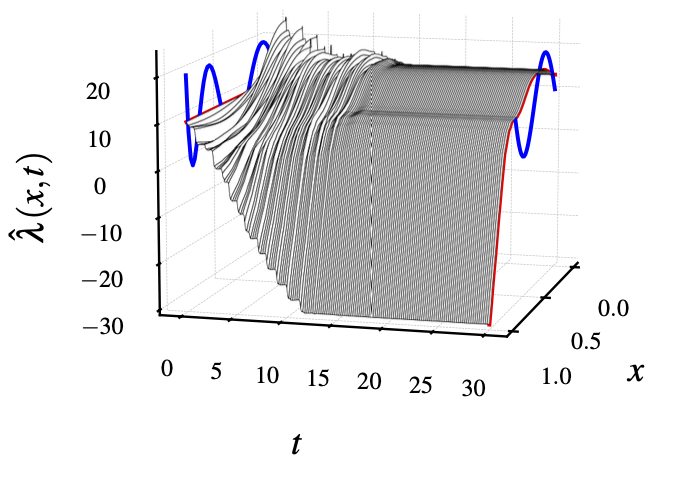}
		\label{figDeepadaptiveETClambda}	
	}	\\
		\subfloat[Results for $\hat{\theta}(t)$ estimate in closed-loop system with adaptive ETC.]{
		\includegraphics [width=4.1cm] {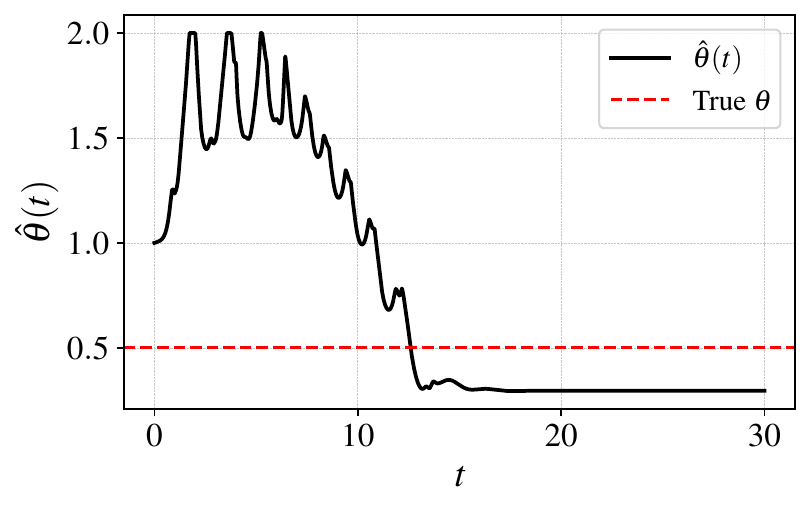}
		\label{figadaptiveETCtheta} }
		\subfloat[Results for $\hat{\theta}(t)$ estimate in closed-loop system with  neural operator-based   adaptive ETC.]{
		\includegraphics [width=4.1cm] {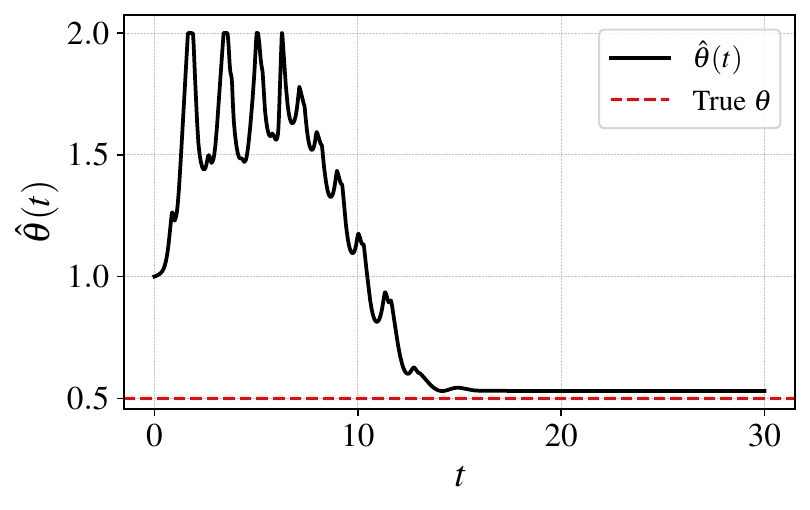}
		\label{figDeepadaptiveETCtheta}	}	
		\caption{Results for parameter estimation.}
	\end{figure}

	\begin{figure}
		\hspace{0.1cm}
		\includegraphics[width=8cm]{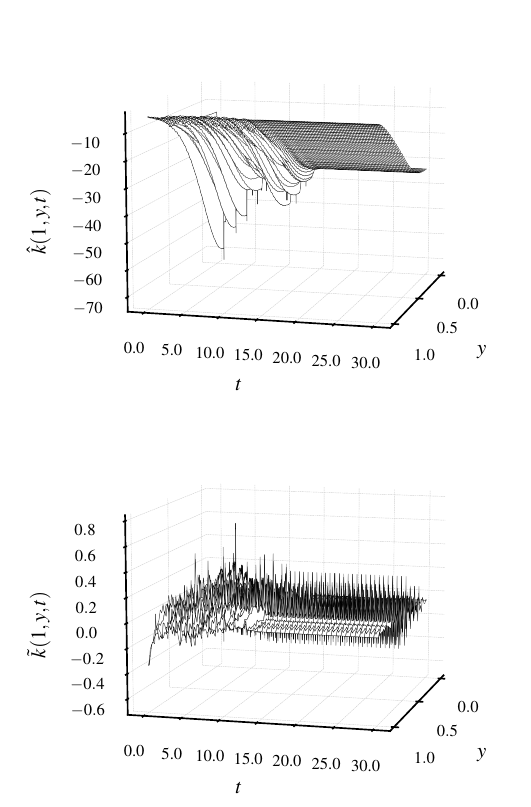} 	
		\caption{Results for $\hat{k}(1,y,t)$ and  $\tilde{k}(1,y,t)$ under neural operator-based adaptive ETC. } 
		\label{figDeepadaptiveETCk}
	\end{figure}

		\begin{table}[t]
		\caption{Calculation times for the finite difference scheme and neural operator averaged over 100 kernel calculations at various spatial resolutions.}
		\label{table_kernel_time}
		\centering
		\footnotesize
		
		\begin{tabular}{cccc}
			\toprule
			Step Size $\Delta x$ & Finite-difference  & DeepONet & Speedup \\
			& calculation time   & calculation time   \\
			& for kernel $k$ ($\mathrm{ms}$) &  for kernel $k$ ($\mathrm{ms}$) \\
			\midrule
			0.04  & 0.53  & 1.10 & $0.48\times$ \\
			0.02  & 2.19  & 0.43 & $5.09\times$ \\
			0.01  & 8.74   & 0.56 & $15.53\times$ \\
			0.005 & 34.87 & 1.97 & $34.87\times$ \\
			\bottomrule
		\end{tabular}

	\end{table}

	\section{Conclusion}\label{sec7}
	This paper presented a switched event-triggered adaptive boundary control framework for reaction-diffusion PDEs with parametric uncertainty. To handle parametric uncertainties, adaptive laws were incorporated, resulting in time-varying gain kernels in the boundary feedback design. To enable efficient implementation, DeepONet-based neural operators were employed to approximate the backstepping kernels, avoiding repeated \textcolor{black}{computation of}   solutions of kernel PDEs. A Lyapunov-based analysis accounting for event-triggering, adaptation, and approximation errors established $L^2$ global asymptotic regulation with arbitrarily small approximation error.

Future work includes the development of an observer-based safe  switched   event-triggered adaptive boundary control  of the proposed framework under actuator constraints as well as its extension to periodic event-triggered and self-triggered control architectures.
	
	\appendix
    	\setcounter{equation}{0}
	\renewcommand{\theequation}{A.\arabic{equation}}
	\subsection{The calculation details in the backstepping transformations \eqref{bw}} \label{Appendix2}
	Differentiating \eqref{bw} with respect to $x$ and $t$, then from \eqref{targ3}, one obtains
	\begin{align}
		&   {w}_{x x}(x,t)-\gamma(x) B(\theta)  {w}(0,t) +  \bigg(-h(x,x)  {w}_{x}(x,t) \notag\\
		&+h_y(x,x)  {w}(x,t)-h_y(x,0)  {w}(0,t)\notag\\
		&-\int_{0}^{x} h_{yy}(x,y)  {w}(y,t)dy\bigg)\notag\\
		&+  w (0,t) \int_{0}^{x}  \gamma(y) B(\theta) h(x,y) dy\notag\\
		&-  \bigg(  {w}_{x x}(x, t)-\frac{d h(x, x)}{d x}  {w}(x, t)-h(x, x)  {w}_x(x, t) \notag\\
		& -h_x(x, x)  {w}(x, t)-\int_0^x h_{x x}(x, y)  {w}(y, t) d y \bigg)  =0 .
	\end{align}
	Differentiating \eqref{bw} with respect to $x$ and $t$, we can also obtain
	\begin{align}
		{\beta}(1,t)= &  {w}(1,t) -\int_{0}^{1} h(1,y)  {w}(y,t)dy ,\label{C.21}\\
		{\beta}(0,t) =& {w}(0,t) , \\
		{\beta}_x(0,t)= &  {w}_x(0,t) -h(0,0)w(0,t), \\
		{\beta}_x(1,t)= & w_x(1,t)-\int_{0}^{1} h_x(1,y) w(y,t) dy\notag\\
		& -h(1,1)w(1,t).
	\end{align}
	Thus, the kernel $h(x,y)$ is found to satisfy \eqref{eq:h1}--\eqref{eq:h4}.
    	\setcounter{equation}{0}
	\renewcommand{\theequation}{B.\arabic{equation}}
	\subsection{Proof of Lemma \ref{boundforkernel_t}} \label{Appendix3}
	According to \eqref{eq:h1}, \eqref{eq:h2} and \eqref{eq:checkh}, $\check h(x,y,t)$ can be written as
	\begin{align}
		\check h(x,y,t)=\phi(x-y,t).\label{eq:hph}
	\end{align}
	Recalling \eqref{eq:h4}, we know $\phi(0,t)=0$.
	It is obtained form \eqref{eq:hph} that 
	\begin{align}
		\check  h_y(x,0,t)=-\phi_x(x,t). \label{eq:hy} 
	\end{align}
	Substituting \eqref{eq:hph}, \eqref{eq:hy} into \eqref{eq:h3} yields
	\begin{align}
		\phi_x(x,t)=g(x,t)-\int_0^x g(y,t)\phi(x-y,t)\,dy,\label{eq:phix}
	\end{align}
	where
	\begin{align}\label{eq:g}
		g(x,t)=\gamma(x;\hat\theta(t))B(\hat\theta(t))
	\end{align}
	is bounded because of \eqref{eq:gamma}, \eqref{eq:B}, and the boundedness of $\hat\theta(t)$. It is thus obtained that $\phi(x,t)$ is bounded by integrating \eqref{eq:phix}, applying Cauthy-Schwarz inequality, and Gr\"onwall's inequality.
	
	Define $\psi(x,t):=\phi_t(x,t)$. Differentiating \eqref{eq:phix} yields
	\begin{align}
		\psi_x(x,t)&=g_t(x,t)
		-\int_0^x g_t(y,t)\phi(x-y,t)dy\notag\\&
		-\int_0^x g(y,t)\psi(x-y,t)dy,\label{eq:psix}
	\end{align}
	with $\psi(0,t)=\phi_t(0,t)=0$.

	According to \eqref{eq:B}, we have
	$\|B_t(\hat\theta(t))\|_{\infty} \le C_B \|\dot{\hat{\theta}}\|_{\infty}.$
	Recalling \eqref{eq:g}, \eqref{eq:gt}, we thus obtain
	$\sup_{x\in[0,1]}\|g_t(x,t)\|
	\le C_g \|\dot{\hat{\theta}}\|_{\infty}$
	for some $C_g>0$.

	Integrating \eqref{eq:psix} and applying Cauthy-Schwarz inequality, we obtain
	\begin{align}
		|\psi(x,t)|
		\le C_1 \|\dot{\hat{\theta}}\|_{\infty}
		+ C_0 \int_0^x |\psi(s,t)|ds,~~x\in[0,1]
	\end{align}
	where $C_0,C_1>0$ are positive constants independent of $t$.
	Applying Gr\"onwall's inequality yields
	$
	\sup_{x\in[0,1]}|\psi(x,t)|
	\le C_2 \|\dot{\hat{\theta}}\|_{\infty}
	$
	for some positive $C_2$.
	It means \eqref{eq88} since $\check h_t(x,y,t)=\psi(x-y,t)$.
 
	\bibliography{reference}
	\bibliographystyle{plain}

\end{document}